\documentclass[a4paper,11pt]{amsart}%,oneside
\pdfoutput=1

\usepackage[ansinew]{inputenc}
\usepackage[danish, english]{babel}
\usepackage{amssymb}
\usepackage{amsmath}
\usepackage{amsthm}
\usepackage{ifthen}
\usepackage{mathtools}

\usepackage{verbatim}

\usepackage{tikz}

\usepackage{todonotes}

\usepackage[final,pdfauthor={Sune Precht Reeh},pdftitle={Transfer and characteristic idempotents for saturated fusion systems}]{hyperref}
\usepackage{bookmark}
\usepackage[abbrev,msc-links]{amsrefs}

\numberwithin{equation}{section}
\numberwithin{figure}{section}
\numberwithin{table}{section}
\allowdisplaybreaks

\theoremstyle{plain}
\newtheorem{mainthmIntro}{Theorem}
\newtheorem{mainthm}{Theorem}

\newtheorem{thm}{Theorem}[section]
\newtheorem{prop}[thm]{Proposition}
\newtheorem{lem}[thm]{Lemma}
\newtheorem{cor}[thm]{Corollary}
\theoremstyle{definition}
\newtheorem{dfn}[thm]{Definition}
\newtheorem{ex}[thm]{Example}
\newtheorem{rmk}[thm]{Remark}

\DeclareMathOperator{\Span}{Span}
\DeclareMathOperator{\Mor}{Mor}
\DeclareMathOperator{\Hom}{Hom}
\DeclareMathOperator{\Aut}{Aut}
\newcommand{\xto}[0]{\xrightarrow}

\newcommand{\x}[0]{\times}
\newcommand{\ox}[0]{\otimes}

\newcommand{\Ø}[0]{\emptyset}
\newcommand{\e}[0]{\varepsilon}
\newcommand{\ph}[0]{\varphi}

\newcommand{\Z}[0]{\mathbb Z}
\newcommand{\Q}[0]{\mathbb Q}
\DeclarePairedDelimiter{\abs}\lvert\rvert
\DeclarePairedDelimiter{\gen}\langle\rangle

\newcommand{\cal}{\mathcal}
\renewcommand{\tilde}{\widetilde}

\renewcommand{\bar}{\overline}

\newcommand{\ccset}[1]{Cl(#1)}
\newcommand{\free}[1]{{\protect\tilde\Omega(#1)}}

\newcommand{\op}[0]{{\textup{op}}}
\newcommand{\lc}[1]{\prescript{#1\!}{}}

\usetikzlibrary{arrows,matrix,decorations,positioning,calc}
\tikzset{dot/.style={circle,fill=black,thick,inner sep=0pt,minimum size=1mm,draw}}
\tikzset{arrow/.style={semithick,>=stealth',shorten >=1pt,shorten <=1pt}}
\tikzset{equal/.style={arrow,double distance=2pt}}

\title{Transfer and characteristic idempotents for saturated fusion systems}
\author[S. P. Reeh]{Sune Precht Reeh}
\address{Department of Mathematics, Massachusetts Institute of Technology, Cambridge, Massachusetts, USA}
\email{reeh@mit.edu}
\subjclass[2010]{}
\thanks{Supported by the Danish National Research Foundation through the Centre for Symmetry and Deformation (DNRF92), and by The Danish Council for Independent Research's Sapere Aude program (DFF – 4002-00224).}

\begin{document}
\begin{abstract}
We construct a well-behaved transfer map from the $p$-local Burnside ring of the underlying $p$-group $S$ to the $p$-local Burnside ring of a saturated fusion system $\cal F$. Using this transfer map, we give new results on the characteristic idempotent of $\cal F$ -- the unique idempotent in the $p$-local double Burnside ring of $S$ satisfying properties of Linckelmann and Webb. We describe this idempotent explicitly both in terms of fixed points and as a linear combination of transitive bisets.
Additionally, using fixed points we determine the map for Burnside rings given by multiplication with the characteristic idempotent, and show that this is the transfer map previously constructed. Applying these results, we show that for every saturated fusion system the ring generated by all (not necessarily idempotent) characteristic elements in the $p$-local double Burnside ring is isomorphic as rings to the $p$-local ``single'' Burnside ring of the fusion system, and we disprove a conjecture by Park-Ragnarsson-Stancu on the composition product of fusion systems.
\end{abstract}
\maketitle

\section{Introduction}
Saturated fusion systems are abstract models for the $p$-local structure of finite groups. The canonical example comes from a finite group $G$ with Sylow $p$-subgroup $S$. The fusion system $\cal F_S(G)$ associated to $G$ (and $S$) is a category whose objects are the subgroups of $S$ and where the morphisms between subgroups are the homomorphisms induced by conjugation by elements of $G$.
As shown by Ragnarsson-Stancu in \citelist{\cite{Ragnarsson} \cite{RagnarssonStancu}}, there is a one-to-one correspondence between the saturated fusion systems on a finite $p$-group $S$ and their associated characteristic idempotents in $A(S,S)_{(p)}$, the $p$-localized double Burnside ring of $S$.

In this paper we introduce a particular transfer map $\pi\colon A(S)_{(p)} \to A(\cal F)_{(p)}$ between Burnside rings for a saturated fusion system $\cal F$ and its underlying $p$-group $S$. This transfer map enables us to calculate the fixed points and coefficients of the characteristic idempotent $\omega_{\cal F}$ for the saturated fusion system $\cal F$ and to give a precise description of the products $\omega_{\cal F}\circ X$ and $X\circ \omega_{\cal F}$ for any element $X$ of the double Burnside ring of $S$. We give an application of these results to a conjecture by Park-Ragnarsson-Stancu on the composition product of saturated fusion systems.

%\medskip
In more detail, we first consider the transfer map for Burnside rings of fusion systems:
The Burnside ring $A(S)$ for a finite $p$-group $S$ is the Grothendieck group formed from the monoid of isomorphism classes of finite $S$-sets, with disjoint union as addition and cartesian product as multiplication.
Let
\[\Phi\colon A(S) \to \prod_{\substack{Q\leq S\\ \text{up to $S$-conj.}}} \Z\]\enlargethispage{.75cm}
be the homomorphism of marks, i.e., the injective ring homomorphism whose $Q$-coordinate $\Phi_Q(X)$ counts the number of fixed points $\abs*{X^Q}$ when $X$ is an $S$-set.
Given a fusion system $\cal F$ on $S$, we say that a finite $S$-set $X$, or a general element of $A(S)$, is \emph{$\cal F$-stable} if the action on $X$ is invariant under conjugation in $\cal F$ -- see section \ref{secBurnsideFusion}. The $\cal F$-stable elements form a subring of $A(S)$ which we call \emph{the Burnside ring of $\cal F$} and denote by $A(\cal F)$.

It is useful to have a canonical way to construct an $\cal F$-stable element from any $S$-set preferably in terms of a transfer map $A(S)\to A(\cal F)$ that plays well with the structure of $A(S)$ as a module over $A(\cal F)$. For the $p$-localized Burnside ring $A(\cal F)_{(p)}$ this paper gives a particular choice of such a map $\pi$ with nice properties and the following simple description in terms of fixed points and the mark homomorphism. In addition, $\pi$ is in fact identical to map $A(S)_{(p)}\to A(S)_{(p)}$ induced by the characteristic idempotent associated to $\cal F$ in the double Burnside ring $A(S,S)_{(p)}$ -- see Corollary \ref{corCharOnesidedStabilization}.
\begin{mainthmIntro}\label{thmStabilizationHomIntro}
Let $\cal F$ be a saturated fusion system on a finite $p$-group $S$. We let $A(\cal F)_{(p)}$ denote the $p$-localized Burnside ring of $\cal F$ as a subring of the $p$-localized Burnside ring $A(S)_{(p)}$ for $S$.
For each $X\in A(S)_{(p)}$, there is a well defined $\cal F$-stable element $\pi(X)\in A(\cal F)_{(p)}$ determined by the fixed point formula
\begin{align*}
\Phi_{Q}(\pi(X)) ={}& \frac 1{\abs{[Q]_{\cal F}}}\sum_{Q'\in [Q]_{\cal F}} \Phi_{Q'}(X),
\end{align*}
which takes the average for each $\cal F$-conjugacy class $[Q]_{\cal F}$ of subgroups $Q\leq S$. The resulting map $\pi\colon A(S)_{(p)} \to A(\cal F)_{(p)}$ is a homomorphism of $A(\cal F)_{(p)}$-modules and restricts to the identity on $A(\cal F)_{(p)}$.
\end{mainthmIntro}

If we apply $\pi$ to the transitive $S$-sets $S/P$ for $P\leq S$, we get elements $\beta_P:=\pi(S/P)$, which form a $\Z_{(p)}$-basis for the $p$-localized Burnside ring $A(\cal F)_{(p)}$ by Proposition \ref{thmPLocalBurnsideBasis}, and where $\beta_P=\beta_Q$ if and only if $P$ and $Q$ are conjugate in $\cal F$. In Proposition \ref{propSameBurnsideGroup}, we show that when $\cal F$ arises from a finite group $G$ with Sylow $p$-subgroup $S$, then the basis elements $\beta_P$ are closely related to the transitive $G$-sets $G/P$ for $P\leq S$, and the $p$-localized Burnside ring $A(\cal F)_{(p)}$ is isomorphic to the part of $A(G)_{(p)}$ where all stabilizers are $p$-subgroups.

The (double) Burnside module $A(S,T)$ is defined for a pair of $p$-groups similarly to the Burnside ring of a group, except that we consider isomorphism classes of $(S,T)$-bisets, which are sets equipped with both a right $S$-action and a left $T$-action that commute with each other. The Burnside module $A(S,T)$ is then the Grothendieck group of the monoid formed by isomorphism classes of finite $(S,T)$-bisets with disjoint union as addition. The $(S,T)$-bisets correspond to sets with a left $(T\x S)$-action, and the transitive bisets correspond to transitive sets $(T\x S)/D$ for subgroups $D\leq T\x S$.
Note that we do not make the usual requirement that the bisets have a free left action, and the results below hold for non-free bisets as well. %In the rare case where left-freeness is actually needed in this paper, we let $A^{\lhd}(S,T)$ denote the submodule of left-free bisets.

For every triple of $p$-groups $S$, $T$, $U$ we have a composition map $\circ \colon A(T,U) \x A(S,T) \to A(S,U)$ given on bisets by $Y\circ X := Y\x_T X = Y\x X/ \sim$
where $(yt,x) \sim (y,tx)$ for all $y\in Y$, $x\in X$, and $t\in T$.
For each $D\leq T\x S$ we have a fixed point homomorphism $\Phi_D\colon A(S,T) \to \Z$, but it is only a homomorphism of abelian groups. An element $X\in A(S,T)$ is still fully determined by the number of fixed points $\Phi_D(X)$ for $D\leq T\x S$.
Subgroups in $T\x S$ of particular interest are the graphs of homomorphisms $\ph\colon P\to T$ for $P\leq S$, where the graph of $\ph\colon P\to T$ is the subgroup $\Delta(P,\ph):=\{(\ph(g),g) \mid g\in P\}$. The transitive $(T\x S)$-set $(T\x S)/\Delta(P,\ph)$ corresponds to a transitive $(S,T)$-biset whose isomorphism class we denote by $[P,\ph]_S^T$. These are precisely the transitive $(S,T)$-bisets where the left action by $T$ is free.

Given a saturated fusion system $\cal F$, a particularly nice class of elements in the $p$-localized double Burnside ring $A(S,S)_{(p)}$ are the \emph{$\cal F$-characteristic} elements, which satisfy the following properties put down by Linckelmann-Webb: An element $X\in A(S,S)_{(p)}$ is $\cal F$-charac\-te\-ri\-stic if it is
\begin{description}
\item[$\cal F$-generated] $X$ is a linear combination of the $(S,S)$-bisets $[P,\ph]_S^S$ where $\ph\colon P\to S$ is a morphism of $\cal F$,
\item[Right $\cal F$-stable] For all $P\leq S$ and $\ph\in \cal F(P,S)$ we have
  $X\circ [P,\ph]_P^S = X\circ [P,id]_P^S$ as elements of $A(P,S)_{(p)}$,
\item[Left $\cal F$-stable] For all $P\leq S$ and $\ph\in \cal F(P,S)$ we have
  $[\ph P,\ph^{-1}]_{S}^P\circ X = [P,id]_{S}^P \circ X$ as elements of $A(P,S)_{(p)}$,
\end{description}
and an additional technical condition to ensure that $X$ is not degenerate.

Ragnarsson and Stancu showed in \citelist{\cite{Ragnarsson} \cite{RagnarssonStancu}} that there is a unique $\cal F$-characteristic idempotent $\omega_{\cal F}\in A(S,S)_{(p)}$ associated to each saturated fusion system $\cal F$, and that it is always possible to recover $\cal F$ from $\omega_{\cal F}$.
In this paper we show that the characteristic idempotent has the following number of fixed points, and we give the decomposition of $\omega_{\cal F}$ into biset orbits:
\begin{mainthmIntro}\label{thmCharIdemStructureIntro}
Let $\cal F$ be a saturated fusion system on a finite $p$-group $S$.
The characteristic idempotent $\omega_{\cal F}\in A(S,S)_{(p)}$ associated to $\cal F$ satisfies:

For all graphs $\Delta(P,\ph)\leq S\x S$ with $\ph\in \cal F(P,S)$, we have
\[\Phi_{\Delta(P,\ph)}(\omega_{\cal F}) = \frac{\abs S}{\abs{\cal F(P,S)}};\]
and $\Phi_{D}(\omega_{\cal F})=0$ for all other subgroups $D\leq S\x S$.
Consequently, if we write $\omega_{\cal F}$ in terms of the transitive bisets in $A(S,S)_{(p)}$, we get the expression
\[\omega_{\cal F} = \hspace{-.3cm}\sum_{\substack{[\Delta(P,\ph)]_{S\x S}\\ \text{with } \ph\in \cal F(P,S)}} \frac {\abs S}{\Phi_{\Delta(P,\ph)}([P,\ph]_S^S)} \Big(\sum_{P\leq Q\leq S}\hspace{-.2cm} \frac{\abs{\{\psi\in \cal F(Q,S) \mid \psi|_P=\ph\}}}{\abs{\cal F(Q,S)}}\cdot \mu(P,Q) \Big)[P,\ph]_S^S,\]
where the outer sum is taken over $(S\x S)$-conjugacy classes of subgroups, and where $\mu$ is the M\"obius function for the poset of subgroups in $S$.
\end{mainthmIntro}

We reach these formulas by showing that the characteristic idempotent $\omega_{\cal F}$ coincides with the element $\beta_{\Delta(S,id)}$ that we get by
the transfer map of Theorem \ref{thmStabilizationHomIntro} to the biset $(S\x S)/\Delta(S,id)$ with respect to the product fusion system $\cal F\x\cal F$ on $S\x S$.
A closer look at how Theorem \ref{thmStabilizationHomIntro} is applied to construct $\beta_{\Delta(S,id)}=\omega_{\cal F}$ gives us a precise description of what happens when other elements are multiplied by $\omega_{\cal F}$:
\begin{mainthmIntro}\label{thmCharIdemMultiplicationIntro}
Let $\cal F_1$ and $\cal F_2$ be saturated fusion systems on finite $p$-groups $S_1$ and $S_2$ respectively, and let $\omega_1\in A(S_1,S_1)_{(p)}$ and $\omega_2\in A(S_2,S_2)_{(p)}$ be the characteristic idempotents.

For every element of the Burnside module $X\in A(S_1,S_2)_{(p)}$, the product $\omega_2\circ X\circ \omega_1$ is right $\cal F_1$-stable and left $\cal F_2$-stable, and satisfies
\[\Phi_{D}(\omega_2\circ X \circ \omega_1) = \frac{1}{\abs{[D]_{\cal F_2\x \cal F_1}}} \sum_{D'\in [D]_{\cal F_2\x \cal F_1}} \Phi_{D'}(X),\]
 for all subgroups $D\leq S_2\x S_1$, where $[D]_{\cal F_2\x \cal F_1}$ is the isomorphism class of $D$ in the product fusion system $\cal F_2\x \cal F_1$ on $S_2\x S_1$.
\end{mainthmIntro}
Let $A(\cal F_1,\cal F_2)_{(p)}$ denote the right $\cal F_1$-stable and left $\cal F_2$-stable elements of $A(S_1,S_2)_{(p)}$. Then the characteristic idempotents $\omega_1$ and $\omega_2$ act trivially on $A(\cal F_1,\cal F_2)_{(p)}$, and Theorem \ref{thmCharIdemMultiplicationIntro} gives a transfer homomorphism of modules over the double Burnside rings $A(\cal F_1,\cal F_1)_{(p)}$ and $A(\cal F_2,\cal F_2)_{(p)}$ as described in Proposition \ref{propDoubleStabilizationHom}.

For a saturated fusion system $\cal F$ on $S$, the double Burnside ring $A(\cal F,\cal F)_{(p)}$ is the subring of $A(S,S)_{(p)}$ consisting of all the elements that are both left and right $\cal F$-stable.
An even smaller subring is the collection of all the elements that are $\cal F$-generated as well as $\cal F$-stable. We denote this subring $A^{\text{char}}(\cal F)_{(p)}$ since a generic $\cal F$-generated, $\cal F$-stable element is actually $\cal F$-characteristic. Hence we have a sequence of inclusions of subrings
\[A^{\text{char}}(\cal F)_{(p)} \subseteq A(\cal F,\cal F)_{(p)}\subseteq A(S,S)_{(p)}.\]
The last inclusion is not unital since $\omega_{\cal F}$ is the multiplicative identity of the first two rings, and $[S,id]_S^S$ is the identity of $A(S,S)_{(p)}$.
According to Proposition \ref{propCharElemBasis}, $A^{\text{char}}(\cal F)_{(p)}$ has a $\Z_{(p)}$-basis consisting of elements $\beta_{\Delta(P,id)}$, which only depend on $P\leq S$ up to $\cal F$-con\-ju\-ga\-tion, and each element of $A^{\text{char}}(\cal F)_{(p)}$, written
\[X= \sum_{\substack{P\leq S\\ \text{up to $\cal F$-conj.}}} c_{\Delta(P,id)} \beta_{\Delta(P,id)},\]
is $\cal F$-characteristic if and only if $c_{\Delta(S,id)}$ is invertible in $\Z_{(p)}$.

For every $(S,S)$-biset $X$, we can quotient out the right $S$-action in order to get $X/S$ as a left $S$-set. Quotienting out the right $S$-action preserves disjoint union and extends to a collapse map $q\colon A(S,S)_{(p)}\to A(S)_{(p)}$, and by restriction to subrings we get maps
\begin{equation*}
  \begin{tikzpicture}
    \matrix[matrix of math nodes] (M) {
        A^{\text{char}}(\cal F)_{(p)} &[1cm] A(\cal F,\cal F)_{(p)} &[1cm] A(S,S)_{(p)} \\[1cm]
            & A(\cal F)_{(p)} & A(S)_{(p)} \\
    };
    \path[arrow]
        (M-1-1) -- node{$\subseteq$} (M-1-2) -- node{$\subseteq$} (M-1-3)
        (M-2-2) -- node{$\subseteq$} (M-2-3)
        (M-1-1) edge[->] (M-2-2)
        (M-1-2) edge[->] (M-2-2)
        (M-1-3) edge[->] (M-2-3)
    ;
  \end{tikzpicture}
\end{equation*}
where $\cal F$-stable bisets are collapsed to $\cal F$-stable sets. In general the collapse map does not respect the multiplication of the double Burnside ring, but combining the techniques of Theorems \ref{thmStabilizationHomIntro} and \ref{thmCharIdemMultiplicationIntro} we show that on $A^{\text{char}}(\cal F)_{(p)}$ the collapse map is not only a ring homomorphism but actually an isomorphism of rings!
\begin{mainthmIntro}\label{thmSingleEmbedsInDoubleIntro}
Let $\cal F$ be a saturated fusion system on a finite $p$-group $S$.

Then the collapse map $q\colon A^{\text{char}}(\cal F)_{(p)} \to A(\cal F)_{(p)}$, which quotients out the right $S$-action, is an isomorphism of rings, and it sends the basis element $\beta_{\Delta(P,id)}$ of $A^{\text{char}}(\cal F)_{(p)}$ to the basis element $\beta_{P}$ of $A(\cal F)_{(p)}$.
\end{mainthmIntro}
This generalizes a similar result for groups where the Burnside ring $A(S)$ embeds in the double Burnside ring $A(S,S)$ with the transitive $S$-set $S/P$ corresponding to the transitive biset $[P,id]_S^S$.
As an immediate consequence of Theorem \ref{thmSingleEmbedsInDoubleIntro} we get an alternative proof that the characteristic idempotent $\omega_{\cal F}$ is unique: Corollary \ref{corCharIdemUnique} shows that $\beta_{\Delta(S,id)}=\omega_{\cal F}$ is the only non-zero idempotent of $A^{\text{char}}(\cal F)_{(p)}$ by proving that $0$ and $S/S$ are the only idempotents of $A(\cal F)_{(p)}$.

The final section of this paper applies Theorem \ref{thmCharIdemMultiplicationIntro} to disprove a conjecture by Park-Ragnarsson-Stancu, \cite{ParkRagnarssonStancu}, on the composition product of fusion systems.
Let $\cal F$ be a saturated fusion system on a $p$-group $S$, and let $\cal H,\cal K$ be saturated fusion subsystems on subgroups $R,T\leq S$ respectively. In the terminology of Park-Ragnarsson-Stancu, we then say that $\cal F$ is the composition product of $\cal H$ and $\cal K$, written $\cal F=\cal H\cal K$, if $S=RT$ and for all subgroups $P\leq T$ it holds that every morphism $\ph\in \cal F(P,R)$ can be written as a composition $\ph=\psi\rho$ where $\psi$ is a morphism of $\cal H$ and $\rho$ is a morphism of $\cal K$.

Park-Ragnarsson-Stancu conjectured that $\cal F=\cal H\cal K$ is equivalent to the following equation of characteristic idempotents:
\begin{equation}\label{eqPRSIntro}
[R,id]_S^R \circ \omega_{\cal F} \circ [T,id]_T^S = \omega_{\cal H}\circ [R\cap T,id]_T^R\circ \omega_{\cal K}
\end{equation}
A special case of the conjecture was proven in \cite{ParkRagnarssonStancu}, in the case where $R=S$ and $\cal K$ is weakly normal in $\cal F$, and the general conjecture was inspired by the group case, where $H,K\leq G$ satisfy $G=HK$ if and only if there is an isomorphism of $(K,H)$-bisets $G\cong H\x_{H\cap K} K$. By direct calculation via Theorem \ref{thmCharIdemMultiplicationIntro} we can now characterize all cases where \eqref{eqPRSIntro} holds:
\begin{mainthmIntro}\label{thmCompositionProductsIntro}
Let $\cal F$ be a saturated fusion system on a $p$-group $S$, and suppose that $\cal H,\cal K$ are saturated fusion subsystems of $\cal F$ on subgroups $R,T\leq S$ respectively.

Then the characteristic idempotents satisfy
\begin{equation}\label{eqCompositionIdemsIntro}
[R,id]_S^R \circ \omega_{\cal F} \circ [T,id]_T^S = \omega_{\cal H}\circ [R\cap T,id]_T^R\circ \omega_{\cal K}
\end{equation}
if and only if $\cal F=\cal H\cal K$ and for all $Q\leq R\cap T$ we have
\begin{equation}\label{eqCompositionHomsIntro}
\abs{\cal F(Q,S)} = \frac{\abs{\cal H(Q,R)}\cdot \abs{\cal K(Q,T)}}{\abs{\Hom_{\cal H\cap \cal K}(Q,R\cap T)}}.
\end{equation}
\end{mainthmIntro}
In particular \eqref{eqCompositionIdemsIntro} always implies $\cal F=\cal H\cal K$, but the converse is not true in general.
In Example \ref{exA6product}, the alternating group $A_6$ gives rise to a composition product $\cal F=\cal H\cal K$ where \eqref{eqCompositionHomsIntro} fails -- hence we get a counter-example to the general conjecture of Park-Ragnarsson-Stancu.

At the same time, Proposition \ref{propNormalCompositionProduct} proves a special case of the conjecture where $\cal K$ is weakly normal in $\cal F$. This is a generalization of the case proved by Park-Ragnarsson-Stancu, as Proposition \ref{propNormalCompositionProduct} does not require that $R=S$.
Finally we suggest a revised definition of composition products (see Definition \ref{dfnRevisedCompositionProd}) with respect to which the conjecture holds in general.

\subsection*{Earlier results on Burnside rings for fusion systems}
An alternative candidate for the Burnside ring of a fusion system, was given by Diaz-Libman in \cite{DiazLibman}. The advantage of the Diaz-Libman definition of the Burnside ring is that it is constructed in close relation to a nice orbit category for the centric subgroups in a saturated fusion system $\cal F$. However, by construction the Burnside ring of Diaz-Libman doesn't see the non-centric subgroup at all, in contrast to the definition of $A(\cal F)$ used in this paper where we have basis elements corresponding to all the subgroups. In Proposition \ref{propSameBurnsideCent}, we compare the two definitions and show that if we quotient out the non-centric part of $A(\cal F)_{(p)}$ we recover the centric Burnside ring of Diaz-Libman, and we relate the basis elements given by Diaz-Libman to the basis elements $\beta_P$ used in this paper.

Theorem \ref{thmStabilizationHomIntro} and the construction of characteristic idempotents in this paper is strongly inspired by an algorithm by Broto-Levi-Oliver.
Originally, in \cite{BrotoLeviOliver}, Broto-Levi-Oliver gave a procedure for constructing a characteristic biset $\Omega$ from a saturated fusion system $\cal F$, and using such a biset, they then constructed a classifying spectrum for $\cal F$. In \cite{Ragnarsson} Ragnarsson took a characteristic biset as constructed by Broto-Levi-Oliver, and proceeded to refine this biset to get an idempotent. This proof used a Cauchy sequence argument in the\linebreak $p$-com\-ple\-tion $A(S,S)^\wedge_p$ of the double Burnside ring in order to show that a characteristic idempotent exists. A later part of \cite{Ragnarsson} showed that $\omega_{\cal F}$ is unique and that in fact $\omega_{\cal F}$ lies in the $p$-localization $A(S,S)_{(p)}$ as a subring of the $p$-completion.
The new construction of $\omega_{\cal F}$ as the element $\beta_{\Delta(S,id)}$ given in this paper takes the original procedure by Broto-Levi-Oliver and refines it in order to construct $\omega_{\cal F}$ directly as an element of $A(S,S)_{(p)}$ -- without needing to work in the $p$-completion. Furthermore, this refined procedure generalizes to give us the transfer map of Theorem \ref{thmStabilizationHomIntro}.

Finally, we note that the formula for the fixed points of $\omega_{\cal F}$ given in Theorem \ref{thmCharIdemStructureIntro} coincides with the work done independently by Boltje-Danz in \cite{BoltjeDanz}. The calculations by Boltje-Danz are done by working in their ghost ring for the double Burnside ring and applying the steps of Ragnarsson's proof for the uniqueness of $\omega_{\cal F}$. This way they are able to calculate what the fixed points of $\omega_{\cal F}$ have to be, assuming that $\omega_{\cal F}$ exists. In this paper, the fixed points follow as an immediate consequence of the way we construct $\beta_{\Delta(S,id)}=\omega_{\cal F}$.

\subsection*{Outline}
Section \ref{secFusSys} recalls the definition and basic properties of saturated fusion systems, and establishes the related notation used throughout the rest of the paper. Section \ref{secBurnside} gives a similar treatment to the Burnside ring of a finite group as well as the Burnside ring for a saturated fusion system. Section \ref{secStabilization} is the first main section of the paper, where we consider the structure of the $p$-localization $A(\cal F)_{(p)}$ of the Burnside ring for a saturated fusion system $\cal F$ on a finite $p$-group $S$. In particular, we construct a stabilization map that sends every finite $S$-set to an $\cal F$-stable element in a canonical way, and we prove Theorem \ref{thmStabilizationHomIntro}.
The other main section, section \ref{secCharIdem}, is subdivided in three parts. In \ref{secDoubleBurnside} we recall the double Burnside ring of a group. In \ref{secConstructCharIdem} we apply the stabilization map above for the fusion system $\cal F\x \cal F$ in order to prove Theorem \ref{thmCharIdemStructureIntro}. In \ref{secCharIdemAction} we prove Theorem \ref{thmCharIdemMultiplicationIntro} and study the strong relation between the stabilization homomorphism of Theorem \ref{thmStabilizationHomIntro} and multiplying with the characteristic idempotent.
In section \ref{secRingOfChar} we prove Theorem \ref{thmSingleEmbedsInDoubleIntro} relating the $\cal F$-characteristic elements to the Burnside ring of $\cal F$.
Finally, section \ref{secCompositionProducts} concerns the composition product of fusion systems and Theorem \ref{thmCompositionProductsIntro}.

\subsection*{Acknowledgements}
First I would like to thank K\'ari Ragnarsson who suggested that I look at \cite{ParkRagnarssonStancu} and the conjecture therein -- which in turn inspired sections \ref{secCharIdemAction} to \ref{secCompositionProducts}.
I also thank Radu Stancu and Serge Bouc for some good and fruitful discussions during my visit to Amiens, in particular thanks to Serge for pointing out the existence of \cite{Gluck}. Thanks go to Matthew Gelvin for corrections to the paper and for help finding the counterexample of section \ref{secCompositionProducts}. Additional thanks go to the anonymous reviewer for a lot of constructive suggestions and comments.
Last but not least, I thank my advisor Jesper Grodal for his suggestions, feedback and support during my time as his student.

\section{Fusion systems}\label{secFusSys}
The next few pages contain a very short introduction to fusion systems, which were originally introduced by Puig under the name ``full Frobenius systems.'' The aim is to introduce the terminology from the theory of fusion systems that will be used in the paper, and to establish the relevant notation. For a proper introduction to fusion systems see, for instance, Part I of ``Fusion Systems in Algebra and Topology'' by Aschbacher, Kessar and Oliver, \cite{AKO}.

\begin{dfn}
A \emph{fusion system} $\cal F$ on a $p$-group $S$, is a category where the objects are the subgroups of $S$, and for all $P,Q\leq S$ the morphisms must satisfy:
\begin{enumerate}
\item Every morphism $\ph\in \Mor_{\cal F}(P,Q)$ is an injective group homomorphism, and the composition of morphisms in $\cal F$ is just composition of group homomorphisms.
\item $\Hom_{S}(P,Q)\subseteq \Mor_{\cal F}(P,Q)$, where
    \[\Hom_S(P,Q) = \{c_s \mid s\in N_S(P,Q)\}\]
    is the set of group homomorphisms $P\to Q$ induced by $S$-conjugation.
\item For every morphism $\ph\in \Mor_{\cal F}(P,Q)$, the group isomorphisms $\ph\colon P\to \ph P$ and $\ph^{-1}\colon \ph P \to P$ are elements of $\Mor_{\cal F}(P,\ph P)$ and $\Mor_{\cal F}(\ph P, P)$ respectively.
\end{enumerate}
We also write $\Hom_{\cal F}(P,Q)$ or just $\cal F(P,Q)$ for the morphism set $\Mor_{\cal F}(P,Q)$; and the group $\cal F(P,P)$ of automorphisms is denoted by $\Aut_{\cal F}(P)$.
\end{dfn}

The canonical example of a fusion system comes from a finite group $G$ with a given $p$-subgroup $S$.
The fusion system of $G$ on $S$, denoted $\cal F_S(G)$, is the fusion system on $S$ where the morphisms from $P\leq S$ to $Q\leq S$ are the homomorphisms induced by $G$-conjugation:
\[\Hom_{\cal F_S(G)}(P,Q) := \Hom_G(P,Q) = \{c_g \mid g\in N_G(P,Q)\}.\]
%Unless otherwise mentioned, we only consider $\cal F_S(G)$ in the case where $S$ is a Sylow $p$-subgroup of $G$.
A particular case is the fusion system $\cal F_S(S)$ consisting only of the homomorphisms induced by $S$-conjugation.

Let $\cal F$ be an abstract fusion system on $S$. We say that two subgroup $P,Q\leq S$ are\linebreak \emph{$\cal F$-con\-ju\-gate}, written $P\sim_{\cal F} Q$, if they are isomorphic in $\cal F$, i.e., there exists a group isomorphism $\ph\in \cal F(P,Q)$.
$\cal F$-conjugation is an equivalence relation, and the set of $\cal F$-conjugates to $P$ is denoted by $[P]_{\cal F}$. The set of all $\cal F$-conjugacy classes of subgroups in $S$ is denoted by $\ccset{\cal F}$.
Similarly, we write $P\sim_S Q$ if $P$ and $Q$ are $S$-conjugate, the $S$-conjugacy class of $P$ is written $[P]_{S}$ or just $[P]$, and we write $\ccset S$ for the set of $S$-conjugacy classes of subgroups in $S$.
Since all $S$-conjugation maps are in $\cal F$, any $\cal F$-conjugacy class $[P]_{\cal F}$ can be partitioned into disjoint $S$-conjugacy classes of subgroups $Q\in [P]_{\cal F}$.

We say that $Q$ is \emph{$\cal F$-} or \emph{$S$-subconjugate} to $P$ if $Q$ is, respectively, $\cal F$- or $S$-conjugate to a subgroup of $P$, and we denote this by $Q\lesssim_{\cal F} P$ or $Q\lesssim_S P$ respectively.
In the case where $\cal F = \cal F_S(G)$, we have $Q\lesssim_{\cal F} P$ if and only if $Q$ is $G$-conjugate to a subgroup of $P$; in this case the $\cal F$-conjugates of $P$ are just those $G$-conjugates of $P$ that are contained in $S$.

A subgroup $P\leq S$ is said to be \emph{fully $\cal F$-normalized} if $\abs{N_S P} \geq \abs{N_S Q}$ for all $Q\in [P]_{\cal F}$; similarly $P$ is \emph{fully $\cal F$-centralized} if $\abs{C_S P} \geq \abs{C_S Q}$ for all $Q\in [P]_{\cal F}$.

\begin{dfn}
A fusion system $\cal F$ on $S$ is said to be \emph{saturated} if the following properties are satisfied for all $P\leq S$:
\begin{enumerate}
\item If $P$ is fully $\cal F$-normalized, then $P$ is fully $\cal F$-centralized, and $\Aut_S(P)$ is a Sylow $p$-subgroup of $\Aut_{\cal F}(P)$.
\item Every homomorphism $\ph\in \cal  F(P,S)$ with $\ph (P)$ fully $\cal F$-centralized extends to a homomorphism $\ph\in \cal F(N_\ph,S)$, where
    \[N_\ph:= \{x\in N_S(P) \mid \exists y\in S\colon \ph\circ c_x = c_y\circ \ph\}.\]
\end{enumerate}
\end{dfn}
The saturation axioms are a way of emulating the Sylow theorems for finite groups; in particular, whenever $S$ is a Sylow $p$-subgroup of $G$, then the Sylow theorems imply that the induced fusion system $\cal F_S(G)$ is saturated (see e.g. \cite{AKO}*{Theorem 2.3}).

In this paper, we shall rarely use the defining properties of saturated fusion systems directly. We shall instead mainly use the following lifting property, which saturated fusion systems satisfy:
\begin{lem}[{\cite{RobertsShpectorov}}]\label{lemNormalizerMap}
Let $\cal F$ be saturated. If $P\leq S$ is fully normalized, then for each $Q\in [P]_{\cal F}$ there exists a homomorphism $\ph\in \cal F(N_S Q, N_S P)$ with $\ph(Q) = P$.
\end{lem}
For the proof, see Lemma 4.5 of \cite{RobertsShpectorov} or Lemma 2.6(c) of \cite{AKO}.

\section{Burnside rings for groups and fusion systems}\label{secBurnside}
In this section we recall the Burnside ring of a finite group $S$ and how to describe its structure in terms of the homomorphism of marks, which embeds the Burnside ring into a suitable ghost ring.
We also recall the Burnside ring of a saturated fusion system $\cal F$, in the sense of \cite{ReehStableSets}, which has a similar mark homomorphism and ghost ring.

Let $S$ be a finite group, not necessarily a $p$-group. Then the isomorphism classes of finite $S$-sets form a semiring with disjoint union as addition and cartesian product as multiplication. The Burnside ring of $S$, denoted $A(S)$, is then defined as the additive Grothendieck group of this semiring, and $A(S)$ inherits the multiplication as well. Given a finite $S$-set $X$, we let $[X]$ denote the isomorphism class of $X$ as an element of $A(S)$.
The isomorphism classes $[S/P]$ of transitive $S$-sets form an additive basis for $A(S)$, and two transitive sets $S/P$ and $S/Q$ are isomorphic if and only if the subgroups $P$ and $Q$ are conjugate in $S$.

For each element $X\in A(S)$ we define $c_P(X)$, with $P\leq S$, to be the coefficients when we write $X$ as a linear combination of the basis elements $[S/P]$ in $A(S)$, i.e.
\[X= \sum_{[P]\in \ccset S} c_P(X) \cdot [S/P],\]
where $\ccset S$ denotes the set of $S$-conjugacy classes of subgroup in $S$. The resulting maps $c_P\colon A(S) \to \Z$ are group homomorphisms, but they are \emph{not} ring homomorphisms.

To describe the multiplication of $A(S)$, it is enough to know the products of basis elements $[S/P]$ and $[S/Q]$. By taking the cartesian product $(S/P)\x (S/Q)$ and considering how it breaks into orbits, one reaches the following double coset formula for the multiplication in $A(S)$:
\begin{equation}\label{eqSingleBurnsideDoubleCoset}
[S/P]\cdot [S/Q] = \sum_{\bar s \in P\backslash S /Q} [S/(P\cap \lc s Q)],
\end{equation}
where $P\backslash S /Q$ is the set of double cosets $PsQ$ with $s\in S$, and $\lc s Q$ is short for $sQs^{-1}$.

Instead of counting orbits, an alternative way of characterising a finite $S$-set is counting the fixed points for each subgroup $P\leq S$. For every $P\leq S$ and $S$-set $X$, we denote the number of $P$-fixed points by $\Phi_{P}(X) := \abs*{X^P}$. This number only depends on $P$ up to $S$-conjugation.
Since we have
\[\abs*{(X \sqcup Y)^P}= \abs*{X^P}+\abs* {Y^P} \quad\text{and}\quad\abs*{(X \x Y)^P}= \abs*{X^P}\cdot\abs*{Y^P}\]
for all $S$-sets $X$ and $Y$, the \emph{fixed point map} $\Phi_{P}$ for $S$-sets extends to a ring homomorphism $\Phi_{P}\colon A(S) \to \Z$.
On the basis elements $[S/P]$, the number of fixed points is given by
\begin{equation}\label{eqPhiOnBasis}
\Phi_{Q}([S/P]) = \abs*{(S/P)^Q}= \frac{\abs{N_S(Q,P)}}{\abs P},
\end{equation}
where $N_S(Q,P) = \{s\in S \mid \lc s Q \leq P\}$ is the transporter in $S$ from $Q$ to $P$.
In particular, $\Phi_{Q}([S/P])\neq 0$ if and only if $Q\lesssim_S P$ ($Q$ is subconjugate to $P$).

\phantomsection\label{dfnFreeS}\label{dfnObsS}
We have one fixed point homomorphism $\Phi_P$ per conjugacy class of subgroups in $S$, and we combine them into the \emph{homomorphism of marks} $\Phi=\Phi^S\colon A(S) \xto{\prod_{[P]} \Phi_{P}} \prod_{[P]\in \ccset S} \Z$. This ring homomorphism maps $A(S)$ into the product ring $\free{S}:=\prod_{[P]\in \ccset S} \Z$, the so-called \emph{ghost ring} for the Burnside ring $A(S)$.

Results by tom Dieck and others show that the mark homomorphism is injective, and that the cokernel of $\Phi$ is the  \emph{obstruction group} $Obs(S) := \prod_{[P]\in \ccset S} (\Z/\abs{W_S P}\Z)$, where $W_S P := N_S P / P$. These statements are combined in the following proposition, the proof of which can be found in \cite{tomDieck}*{Chapter 1}, \cite{Dress}, and \cite{YoshidaSES}.
\begin{prop}\label{propYoshidaGroup}
Let $\Psi=\Psi^S\colon \free S \to Obs(S)$ be given by the $[P]$-coordinate functions
\[\Psi_{P}(\xi)  := \sum_{\bar s\in W_S P} \xi_{\gen s P} \pmod {\abs{W_S P}}.\]
Here $\xi_{\gen s P}$ denotes the $[\gen s P]$-coordinate of an element $\xi\in \free S = \prod_{[P]\in \ccset S} \Z$.

The following sequence of abelian groups is then exact:
\[0\to A(S) \xto{\Phi} \free S \xto{\Psi} Obs(S) \to 0.\]
$\Phi$ is a ring homomorphism, but $\Psi$ is just a group homomorphism.
\end{prop}
The homomorphism of marks enables us to perform calculations for the Burnside ring $A(S)$ inside the much nicer product ring $\free S$, where we identify each element $X\in A(S)$ with its fixed point vector $(\Phi_Q(X))_{[Q]\in \ccset S}$.

\subsection{The Burnside ring of a saturated fusion system}\label{secBurnsideFusion}
Let $S$ be a finite $p$-group, and suppose that $\cal F$ is a saturated fusion system on $S$. We say that a finite $S$-set is \emph{$\cal F$-stable} if the action is unchanged up to isomorphism whenever we act through morphisms of $\cal F$.
More precisely, if $P\leq S$ is a subgroup and $\ph\colon P\to S$ is a homomorphism in $\cal F$, we can turn $X$ into a $P$-set by using $\ph$ to define the action $g.x:=\ph(g)x$ for $g\in P$. We denote the resulting $P$-set by $\prescript{}{P,\ph}X$. In particular when $incl\colon P\to S$ is the inclusion map, $\prescript{}{P,incl}X$ has the usual restriction of the $S$-action to $P$.
Restricting the action of $S$-sets along $\ph$ extends to a ring homomorphism $r_\ph\colon A(S)\to A(P)$, and we let $\prescript{}{P,\ph}X$ denote the image $r_\ph(X)$ for all elements $X\in A(S)$.

We then say that an element $X\in A(S)$ is \emph{$\cal F$-stable} if it satisfies
\begin{equation}\label{charFstable}
\parbox[c]{.9\textwidth}{\emph{$\prescript{}{P,\ph}X=\prescript{}{P,incl}X$ inside $A(P)$, for all $P\leq S$ and homomorphisms $\ph\colon P\to S$ in $\cal F$.}}
\end{equation}
Alternatively, one can characterize $\cal F$-stability in terms of fixed points and the mark homomorphism, and the following three properties are equivalent for all $X\in A(S)$:
\begin{enumerate}
\item\label{itemPhiBurnsideEq}  $X$ is $\cal F$-stable.
\item\label{itemPhiStable} $\Phi_{P}(X) = \Phi_{\ph P}(X)$ for all $\ph\in \cal F(P,S)$ and $P\leq S$.
\item\label{itemFConjStable} $\Phi_{P}(X) = \Phi_{Q}(X)$ for all pairs $P,Q\leq S$ with $P\sim_{\cal F} Q$.
\end{enumerate}
A proof of this claim can be found in \cite{Gelvin}*{Proposition 3.2.3} or \cite{ReehStableSets}. We shall primarily use \ref{itemPhiStable} and \ref{itemFConjStable} to characterize $\cal F$-stability.

It follows from property \ref{itemFConjStable} that the $\cal F$-stable elements form a subring of $A(S)$. We define the \emph{Burnside ring of $\cal F$} to be the subring $A(\cal F)\subseteq A(S)$ consisting of all the $\cal F$-stable elements. Equivalently, we can consider the actual $S$-sets that are $\cal F$-stable: The $\cal F$-stable sets form a semiring, and we define $A(\cal F)$ to be the Grothendieck group hereof. These two constructions give rise to the same ring $A(\cal F)$ -- see \cite{ReehStableSets}*{Proposition 4.4}.
As is the case for the Burnside ring of a group, $A(\cal F)$ has an additive basis, where the basis elements are in one-to-one correspondence with the $\cal F$-conjugacy classes of subgroups in $S$.

For each $X\in A(\cal F)$ the fixed point map $\Phi_{P}(X)$ only depends on $P$ up to $\cal F$-conjugation. The homomorphism of marks for $A(S)$ therefore restricts to the subring $A(\cal F)$ as an injective ring homomorphism
\[\Phi^{\cal F} \colon A(\cal F) \xto{\prod_{[P]_{\cal F}} \Phi_{P}} \prod_{[P]_{\cal F}\in \ccset {\cal F}} \Z,\]
where $\ccset{\cal F}$ denotes the set of $\cal F$-conjugacy classes of subgroups in $S$. The product ring $\free{\cal F} := \prod_{[P]_{\cal F}\in \ccset {\cal F}} \Z$ we call the \emph{the ghost ring} for $A(\cal F)$, and we view $\free{\cal F}$ as a subring of $\free{S}$ consisting of the vectors that are constant on each $\cal F$-conjugacy class of subgroups.
The ring homomorphism $\Phi^{\cal F}$ is called \emph{the homomorphism of marks} for $A(\cal F)$.

As for the Burnside ring of a group, we also have an explicit description of the cokernel of $\Phi^{\cal F}$ as the group
\[Obs(\cal F) := \prod_{\substack{[P]\in\ccset{\cal F}\\ P\text{ f.n.}}} (\Z / \abs{W_S P}\Z),\]
where $P$ is taken to be a fully normalized representative for each $\cal F$-conjugacy class of subgroups.
According to \cite{ReehStableSets}*{Theorem B}, we have a short-exact sequence similar to Proposition \ref{propYoshidaGroup}:
\begin{prop}\label{propYoshidaFusion}
Let $\Psi=\Psi^{\cal F}\colon \free {\cal F} \to Obs(\cal F)$ be given by the $[P]_{\cal F}$-coordinate functions
\[\Psi_{P}(\xi)  := \sum_{\bar s\in W_S P} \xi_{\gen s P} \pmod {\abs{W_S P}},\]
when $P$ is fully $\cal F$-normalized, and $\xi_{\gen s P}$ denotes the $[\gen s P]_{\cal F}$-coordinate of an element $\xi\in \free {\cal F} = \prod_{[P]\in \ccset {\cal F}} \Z$.

The following sequence of abelian groups is then exact:
\[0\to A(\cal F) \xto{\Phi} \free {\cal F} \xto{\Psi} Obs(\cal F) \to 0.\]
$\Phi$ is a ring homomorphism, but $\Psi$ is just a group homomorphism.
\end{prop}

\section{The $p$-localized Burnside ring}\label{secStabilization}
Let $\cal F$ be a saturated fusion system on a $p$-group $S$. In this section we describe a well-defined stabilization map $A(S)_{(p)}\to A(\cal F)_{(p)}$ between $p$-localized Burnside rings. We will later see in Corollary \ref{corCharOnesidedStabilization} that this map is in fact induced by the characteristic idempotent of $\cal F$ though that is not how we initially construct it. The map is shown to be a homomorphism of $A(\cal F)_{(p)}$-modules, and it has a simple expression in terms of the mark homomorphism for $A(S)_{(p)}$ as stated in Theorem \ref{thmStabilizationHom}. Using this stabilization homomorphism, we give a new basis for $A(\cal F)_{(p)}$. It was shown in \cite{ReehStableSets} that the irreducible $\cal F$-stable sets form a basis for $A(\cal F)$, but very little is known about their actual structure. The new basis for $A(\cal F)_{(p)}$, though it only exists after $p$-localization, is easily described in terms of the homomorphism of marks.
We use this basis in section \ref{secCharIdem}, for the product fusion system $\cal F\x \cal F$ on $S\x S$, to give a new description of the so-called \emph{characteristic idempotent} for the saturated fusion system $\cal F$.
In section \ref{secEquivBurnsideRings} we compare $A(\cal F)_{(p)}$, including its basis, with the centric Burnside ring of $\cal F$ defined by Diaz and Libman in \cite{DiazLibman}. When $\cal F$ is realized by a group $G$, we also relate $A(\cal F)_{(p)}$ to the $p$-subgroup part of $A(G)_{(p)}$.

\medskip
It is useful to have a procedure for constructing an $\cal F$-stable set from a general $S$-set. Such a procedure was used by Broto, Levi and Oliver in \cite{BrotoLeviOliver} to show that every saturated fusion system has at least one ``characteristic biset,'' a set with left and right $S$-actions satisfying properties suggested by Linckelmann and Webb. A similar procedure was used in \cite{ReehStableSets}, to construct all irreducible $\cal F$-stable $S$-sets.
Both constructions follow the same general idea: To begin with, we are given a finite $S$-set $X$ (or in general an element of the Burnside ring). We then consider each $\cal F$-conjugacy class of subgroups in $S$ in decreasing order and add further $S$-orbits to $X$ until the set becomes $\cal F$-stable. To construct the irreducible $\cal F$-stable sets, we start with a transitive $S$-set $[S/P]$; to construct a characteristic biset, we start with $S$ itself considered as an $(S,S)$-biset.

The construction changes the number of elements and orbits in the set $X$ that we stabilize, and the number of added orbits depends heavily on the set that we start with -- if $X$ is already $\cal F$-stable we need not add anything at all.
Because of this, we expect the stabilized sets to behave quite differently from the sets we start with, for instance, the stabilization procedure does not even preserve addition.

In this section we adjust the construction of \citelist{\cite{BrotoLeviOliver} \cite{ReehStableSets}} such that instead of just adding orbits to stabilize a set, we subtract orbits as well, in a way such that all changes cancel ``up to $\cal F$-conjugation.'' This results in a nicely behaved stabilization procedure that works for all $S$-sets, with one disadvantage: we must work in the $p$-localization $A(S)_{(p)}$ instead of $A(S)$.

The following Lemmas \ref{lemNumberOfConjugates} and \ref{lemWeightedMeanOfCoordinates} are needed to show that the later calculations work in $A(S)_{(p)}$, i.e., that we never divide by $p$. Lemma \ref{lemNumberOfConjugates} is also interesting in itself since it shows that for any fully normalized subgroup $P\leq S$, the number of $\cal F$-conjugates to $P$ is the same as the number of $S$-conjugates up to a $p'$-factor.
\begin{lem}\label{lemNumberOfConjugates}
Let $\cal F$ be a saturated fusion system on $S$, and let $P\leq S$ be fully $\cal F$-normal\-ized. Then the number of $\cal F$-conjugates of $P$ can be expressed as $\frac{\abs S}{\abs{N_S P}}\cdot k$ where $p\nmid k$. Under the same assumptions we also have $\abs{\cal F(P,S)} = \frac{\abs S}{\abs{C_S P}}\cdot k'$ with $p\nmid k'$.
\end{lem}

\begin{proof}
Recall that $[P]_{\cal F}$ denotes the set of subgroups in $S$ that are $\cal F$-conjugate to $P$.
We then have $\abs{\cal F(P,S)} = \abs{\Aut_{\cal F}(P)} \cdot \abs{[P]_{\cal F}}$ for all $P\leq S$. When $P$ is fully $\cal F$-normalized, we furthermore get
\[\abs{\Aut_{\cal F}(P)} = \abs{\Aut_S(P)}\cdot k'' = \tfrac{\abs{N_S P}}{\abs{C_S P}} \cdot k''\] where $p\nmid k''$ since $\cal F$ is saturated. If we prove that $\abs{[P]_{\cal F}} =\frac{\abs S}{\abs{N_S P}}\cdot k$ where $p\nmid k$, then it immediately follows that
\[ \abs{\cal F(P,S)} = \abs{\Aut_{\cal F}(P)} \cdot \abs{[P]_{\cal F}} = \tfrac{\abs S}{\abs{C_S P}} \cdot k''k\] where $p\nmid k''k$. It is therefore sufficient to prove the claim about $\abs{[P]_{\cal F}}$.

The $\cal F$-conjugacy class $[P]_{\cal F}$ is a disjoint union of the $S$-conjugacy classes $[Q]_S$ where $Q\sim_{\cal F} P$. The $S$-conjugacy class $[Q]_S$ has $\abs S / \abs{N_S Q}$ elements; and $\frac{\abs S}{\abs{N_S Q}}$ is divisible by $\frac{\abs S}{\abs{N_S P}}$ since $P$ is fully normalized. In particular, $\frac{\abs S}{\abs{N_S P}}$ divides $\abs{[P]_{\cal F}}$.

Furthermore, we have $\abs[\big]{[Q]_S}\cdot \frac{\abs{N_S P}}{\abs S}= \frac{\abs{N_S P}}{\abs{N_S Q}}\equiv 0 \pmod p$ whenever $Q\sim_{\cal F} P$ isn't fully normalized.
It follows that
\begin{align*}
\abs[\big]{[P]_{\cal F}}\cdot \frac{\abs{N_S P}}{\abs S} &=\sum_{[Q]_S\subseteq [P]_{\cal F}} \abs[\big]{[Q]_S}\cdot \frac{\abs{N_S P}}{\abs S}
\\ &\equiv \sum_{\substack{[Q]_S\subseteq [P]_{\cal F}\\ Q\text{ f.n.}}} \abs[\big]{[Q]_S}\cdot \frac{\abs{N_S P}}{\abs S} = \abs[\big]{[P]_{\cal F}^{\text{f.n.}}}\cdot \frac{\abs{N_S P}}{\abs S} \pmod{p},
\end{align*}
where ``f.n.'' is short for ``fully normalized,'' and $[P]_{\cal F}^{\text{f.n.}}$ is the set of $Q\sim_{\cal F} P$ that are fully normalized. We conclude that $\abs[\big]{[P]_{\cal F}}= \frac{\abs S}{\abs{N_S P}}\cdot k$ with $p\nmid k$ if and only if $\abs[\big]{[P]_{\cal F}^{\text{f.n.}}} = \frac{\abs S}{\abs{N_S P}}\cdot u$ with $p\nmid u$.

The number $u$ is the number for $S$-conjugacy classes of fully normalized subgroups in $[P]_{\cal F}^{\text{f.n.}}$, as these all have the same size $\frac {\abs S}{\abs{N_S P}}$. The fact $p\nmid u$ is then the precise statement of \cite{BCGLO2}*{Proposition 1.16} which was proved using the existence of characteristic bisets for $\cal F$. It is possible to prove the claim directly, though it becomes slightly complicated.
\end{proof}

\begin{lem}\label{lemWeightedMeanOfCoordinates}
Let $P,Q\leq S$, then $\abs{[Q]_{\cal F}}$ divides $\abs{[Q']_S}$ in $\Z_{(p)}$ for all $Q'\sim_\cal F Q$; and furthermore
\[\frac 1{\abs{[Q]_{\cal F}}}\sum_{Q'\in [Q]_{\cal F}}\Phi_{Q'}([S/P]) = \sum_{[Q']_{S}\subseteq [Q]_{\cal F}}\frac {\abs{[Q']_S}}{\abs{[Q]_{\cal F}}}\Phi_{Q'}([S/P]) =\frac{\abs{\cal F(Q,P)}\cdot \abs S}{\abs P\cdot \abs{\cal F(Q,S)}}\in \Z_{(p)}.\]
\end{lem}

\begin{proof}
By Lemma \ref{lemNumberOfConjugates} we can express the number of $\cal F$-conjugates as $\abs{[Q]_{\cal F}} = \frac{\abs{S}}{\abs{N_S Q_0}} \cdot k$, with $p\nmid k$, where $Q_0\sim_{\cal F} Q$ is fully normalized. At the same time, the number of $S$-con\-ju\-gates of $Q'$ is given by $\abs{[Q']_S} = \frac{\abs{S}}{\abs{N_S Q'}}$. Since $\abs{N_S Q'}\leq \abs{N_S Q_0}$, it then follows that $\abs{[Q]_{\cal F}}$ divides $\abs{[Q']_S}$ in $\Z_{(p)}$.

We try to simplify the sum in the lemma:
\begin{align*}
\sum_{[Q']_S \subseteq [Q]_{\cal F}} \frac{\abs{[Q']_S}}{\abs{[Q]_{\cal F}}}\Phi_{Q'}([S/P]) &= \frac{1}{\abs{[Q]_{\cal F}}}\sum_{[Q']_S \subseteq [Q]_{\cal F}} \frac{\abs{S}}{\abs{N_S(Q')}}\cdot\frac{\abs{N_S(Q', P)}}{\abs P}
\\ &= \frac{\abs{S}}{\abs P \cdot \abs{[Q]_{\cal F}}}\sum_{[Q']_S \subseteq [Q]_{\cal F}} \frac{\abs{N_S(Q', P)}}{\abs{N_S(Q')}}
\\ &= \frac{\abs{S}}{\abs P \cdot \abs{[Q]_{\cal F}}}\sum_{[Q']_S \subseteq [Q]_{\cal F}} \abs{\{R\in [Q']_S \mid R\leq P\}}
\\ &= \frac{\abs{S}}{\abs P \cdot \abs{[Q]_{\cal F}}} \abs{\{R\in [Q]_{\cal F} \mid R\leq P\}}
\\ &= \frac{\abs{\cal F(Q,P)}\cdot\abs{S}}{\abs P \cdot \abs{\cal F(Q,S)}}.
\end{align*}
The last equality follows from multiplying with $\abs{\Aut_{\cal F} (Q)}$ in both the numerator and the denominator.
\end{proof}

Given any element $X$ in the $p$-localized Burnside ring $A(S)_{(p)}$, we stabilize $X$ according to the following idea: We run through the subgroups $Q\leq S$ in decreasing order and subtract/add orbits to $X$ such that it becomes $\cal F$-stable at the conjugacy class of $Q$ in $\cal F$, i.e., such that $\Phi_{Q'}(X)=\Phi_Q(X)$ for all $Q'\sim_{\cal F} Q$. Here we take care to ``add as many orbits as we remove'' at each step.
The actual work of the stabilization procedure is handled in the following technical Lemma \ref{lemInduceFromASptoAFp}, which is then applied in Theorem \ref{thmStabilizationHom} to construct the stabilization map $A(S)_{(p)} \to A(\cal F)_{(p)}$.
Recall that $c_P(X)$ denotes the coefficient of $[S/P]$ when $X$ is written in the standard basis of $A(S)_{(p)}$, and $\Phi_P\colon A(S)_{(p)} \to \Z_{(p)}$ for $P\leq S$ denotes the fixed point homomorphisms.

\begin{dfn}\label{dfnStabilization}
Let $\cal F$ be a saturated fusion system on a $p$-group $S$, and let $\cal H$ be a collection of subgroups of $S$ such that $\cal H$ is closed under taking $\cal F$-subconjugates, i.e., if $P\in \cal H$, then $Q\in \cal H$ for all $Q\lesssim_{\cal F} P$.

An element $X\in A(S)_{(p)}$ is said to be \emph{$\cal F$-stable away from $\cal H$} if $\Phi_{Q}(X) = \Phi_{Q'}(X)$ for all pairs $Q\sim_{\cal F}Q'$, with $Q,Q'\not\in \cal H$.

Given an $X\in A(S)_{(p)}$ that is $\cal F$-stable away from $\cal H$ and another element $Y\in A(S)_{(p)}$, we say that \emph{$Y$ is an $\cal F$-stabilization of $X$ over $\cal H$} if the following properties are satisfied:
\begin{enumerate}
\item\label{itemInducePreserve} $\Phi_{Q}(Y)=\Phi_{Q}(X)$ and $c_{Q}(Y) = c_{Q}(X)$ for all $Q\not\in\cal H$, $Q\leq S$.
\item\label{itemInduceCoord} For all $Q\leq S$ we have
\[\sum_{[Q']_S \subseteq [Q]_{\cal F}} c_{Q'}(Y) = \sum_{[Q']_S \subseteq [Q]_{\cal F}} c_{Q'}(X).\]
\item\label{itemInduceWeightedMean} For every $Q\leq S$:
\begin{align*}
\Phi_{Q}(Y) %&= \frac1{\abs{\cal F(Q,S)}} \sum_{\ph\in \cal F(Q,S)} \Phi_{\ph Q}(X) \\
= \sum_{[Q']_S\subseteq [Q]_{\cal F}} \frac{\abs{[Q']_S}}{\abs{[Q]_{\cal F}}} \Phi_{Q'}(X)
= \frac{1}{\abs{[Q]_{\cal F}}} \sum_{Q'\in  [Q]_{\cal F}}  \Phi_{Q'}(X).
\end{align*}
\end{enumerate}
Here $[Q]_{\cal F}$ denotes the set of $\cal F$-conjugates of $Q$. In the sums we pick one representative $Q'$ for each $S$-conjugacy class $[Q']_S$ contained in $[Q]_{\cal F}$, and by Lemma \ref{lemWeightedMeanOfCoordinates} the fractions $\frac{\abs{[Q']_S}}{\abs{[Q]_{\cal F}}}$ make sense in $\Z_{(p)}$.
\end{dfn}

\begin{rmk}\label{rmkStabilization}
If an $\cal F$-stabilization $Y$ of $X$ over $\cal H$ exists, then
property \ref{itemInducePreserve} ensures that $Y$ retains the part of $X$ that has already been stabilized. Property \ref{itemInduceCoord} is the requirement that the total number of orbits is preserved from $X$ to $Y$ for each $\cal F$-conjugacy class of subgroups. $Y$ is only allowed to ``replace'' an orbit $[S/Q]$ by another orbit $[S/Q']$ where $Q'\sim_{\cal F} Q$. Finally, property \ref{itemInduceWeightedMean} tells us exactly what happens to the mark homomorphism on $Y$: We simply take the average of the fixed points of $X$ for each conjugacy class in $\cal F$. Hence $Y$ is unique if it exists, and because the right hand side of \ref{itemInduceWeightedMean} only depends on $Q$ up to $\cal F$-conjugation, $Y$ must be $\cal F$-stable.
Property \ref{itemInduceWeightedMean} also implies that such an $\cal F$-stabilization $Y$ of $X$ is independent of the choice of collection $\cal H$, as long as the chosen collection $\cal H$ satisfies the assumptions above.
\end{rmk}

\begin{lem}
\label{lemInduceFromASptoAFp}
Let $\cal F$ be a saturated fusion system on a $p$-group $S$, and let $\cal H$ be a collection of subgroups of $S$ such that $\cal H$ is closed under taking $\cal F$-subconjugates.
Assume that $X\in A(S)_{(p)}$ is $\cal F$-stable away from $\cal H$.

Then there exists a uniquely determined element $\pi_{\cal H} X\in A(\cal F)_{(p)}\subseteq A(S)_{(p)}$ which is an $\cal F$-stabilization of $X$ over $\cal H$.
\end{lem}

\begin{proof}
We proceed by induction on the size of $\cal H$, and show that every $X$ that is $\cal F$-stable away from $\cal H$ has an $\cal F$-stabilization over $\cal H$ denoted $\pi_{\cal H} X$. As remarked above, the property \ref{itemInduceWeightedMean} implies that $\pi_{\cal H} X$ is unique so it is sufficient to show that $\pi_{\cal H} X$ exists.

If $\cal H = \Ø$, then $X$ is $\cal F$-stable by assumption. The first two properties for an $\cal F$-stabilization over $\cal H$ are vacuously true for $\pi_{\cal H} X := X$. Furthermore, since $X$ is $\cal F$-stable, we have $\Phi_{Q}(X) = \Phi_{Q'}(X)$ for all pairs $Q\sim_{\cal F}Q'$, and therefore
\[\frac{1}{\abs{[Q]_{\cal F}}}\sum_{Q'\in [Q]_{\cal F}} \Phi_{Q'}(X)
= \frac{\Phi_{Q}(X)}{\abs{[Q]_{\cal F}}}\sum_{Q'\in [Q]_{\cal F}} 1  = \Phi_{Q}(X).\]
Hence $\pi_{\cal H} X = X$ is an $\cal F$-stabilization of $X$ over $\cal H=\Ø$.

Now assume that $\cal H\neq \Ø$, and let $P\in \cal H$ be maximal under $\cal F$-subconjugation as well as fully normalized. Furthermore assume that the lemma is true for all collections strictly smaller than $\cal H$.

Let $P'\sim_{\cal F} P$. Then there is a homomorphism $\ph\in \cal F(N_S P', N_S P)$ with $\ph(P')=P$ by Lemma \ref{lemNormalizerMap} since $\cal F$ is saturated. The restriction of $S$-actions to the subgroup $\ph(N_S P')$ gives a ring~homo\-morphism $A(S)_{(p)} \to A(\ph(N_S P'))_{(p)}$ that preserves the fixed-point homomorphisms $\Phi_{Q}$ for $Q\leq \ph(N_S P')\leq N_S P$.

If we consider $X$ as an element of $A(\ph(N_S P'))_{(p)}$, we can apply the short exact sequence of Proposition \ref{propYoshidaGroup} to
get $\Psi^{\ph(N_S P')}(\Phi(X))=0$. In particular, the $P$-coordinate function satisfies $\Psi^{\ph(N_S P')}_{P}(\Phi(X))=0$, that is,
\[\sum_{\bar s\in \ph(N_S P')/ P} \Phi_{\gen s P}(X) \equiv 0 \pmod{\abs{\ph(N_S P')/P}}.\]
Similarly, we have $\Psi^S(\Phi^S(X))=0$, where the $P'$-coordinate $\Psi^S_{P'}(\Phi^S(X))=0$ gives us
\[\sum_{\bar s\in N_S P'/P'} \Phi_{\gen s P'}(X) \equiv 0 \pmod{\abs{N_S P'/P'}}.\]
Since $P$ is maximal in $\cal H$, we have by assumption $\Phi_{Q}(X) = \Phi_{Q'}(X)$ for all $Q\sim_{\cal F} Q'$ where $P$ is $\cal F$-conjugate to a \emph{proper} subgroup of $Q$. Specifically, we have
\[\Phi_{\gen {\ph (s)}P}(X) = \Phi_{\ph(\gen sP')}(X) = \Phi_{\gen sP'}(X)\] for all $s\in N_S P'$ with $s\not\in P'$. Taking the sum over all nontrivial cosets $\bar s\in N_S P' / P'$, we get the equality
\[\sum_{\substack{\bar s\in \ph(N_S P')/P\\ \bar s \neq \bar 1}} \Phi_{\gen s P}(X) = \sum_{\substack{\bar s\in N_S P' / P'\\ \bar s \neq \bar 1}} \Phi_{\gen s P'}(X).\]
The difference $\Phi_{P}(X)-\Phi_{P'}(X)$ can then be rewritten by adding the left hand side above and subtracting the right hand side:
\begin{align*}
\Phi_{P}(X)-\Phi_{P'}(X)&= \sum_{\bar s\in \ph(N_S P')/P} \Phi_{\gen s P}(X) - \sum_{\bar s\in N_S P' / P'} \Phi_{\gen s P'}(X)
\\ &\equiv 0-0 \pmod{\abs{W_SP'}}.
\end{align*}
We can therefore define
\[\lambda_{P'} := \frac{\Phi_{P}(X)-\Phi_{P'}(X)}{\abs{W_SP'}}\in \Z_{(p)}.\]
We now recall from Lemma \ref{lemNumberOfConjugates} that $\abs[\big]{[P]_{\cal F}} = \frac{\abs S}{\abs{N_S P}} \cdot k$ where $p\nmid k$, and since $k$ is invertible in $\Z_{(p)}$, we can define
\[c := \left(\sum_{[P']_S\subseteq [P]_{\cal F}} \lambda_{P'} \right) \Big/ k\quad \in \Z_{(p)},\]
as well as
\[\mu_{P'} := \lambda_{P'} - \frac{\abs{W_S P}}{\abs{W_S P'}} c \in \Z_{(p)}.\]
We use the $\mu_{P'}$ as coefficients to construct a new element
\begin{equation}\label{eqX'def}
X':=X + \sum_{[P']_S\subseteq [P]_{\cal F}} \mu_{P'}\cdot [S/P'] \in A(S)_{(p)}.
\end{equation}
We then at least have $c_{Q}(X') = c_{Q}(X)$ for all $Q\not\sim_{\cal F} P$.
The definition of $c$ ensures that
\begin{align*}
 \sum_{[P']_S\subseteq [P]_{\cal F}} \frac{\abs{W_S P}}{\abs{W_S P'}} c
 &= c\cdot\sum_{[P']_S\subseteq [P]_{\cal F}} \frac{\abs{N_S P}}{\abs{N_S P'}}
= c\cdot\frac{\abs {N_S P}}{\abs S}\sum_{[P']_S\subseteq [P]_{\cal F}} \abs{[P']_S}
\\ &= c\cdot\frac{\abs {N_S P}}{\abs S}\cdot\abs{[P]_{\cal F}}= c\cdot k
= \sum_{[P']_S\subseteq [P]_{\cal F}} \lambda_{P'};
\end{align*}
which in turn gives us
\begin{equation}\label{eqX'coord}
\begin{split}
&\sum_{[P']_S \subseteq [P]_{\cal F}} c_{P'}(X') - \sum_{[P']_S \subseteq [P]_{\cal F}} c_{P'}(X) = \sum_{[P']_S\subseteq [P]_{\cal F}} \mu_{P'}
\\={}&\sum_{[P']_S\subseteq [P]_{\cal F}} \lambda_{P'}-\sum_{[P']_S\subseteq [P]_{\cal F}} \frac{\abs{W_S P}}{\abs{W_S P'}} c =0.
\end{split}
\end{equation}
Next we recall that $\Phi_Q([S/P']) =0$ unless $Q\lesssim_S P'$, which implies that $\Phi_Q(X') = \Phi_Q(X)$ for every $Q\not\in \cal H$. We then calculate $\Phi_{Q}(X')$ for each $Q\sim_{\cal F} P$:
\begin{equation}\label{eqX'Phi}
\begin{split}
\Phi_{Q}(X') &= \Phi_{Q}(X) + \sum_{[P']_S\subseteq [P]_{\cal F}} \mu_{P'}\cdot\Phi_{Q}([S/P'])
\\ &=  \Phi_{Q}(X) + \mu_{Q}\cdot\Phi_{Q}([S/Q]) \hspace{1.5cm} \text{since $\Phi_{Q}([S/P'])=0$ unless $P'=Q$}
\\ &= \Phi_{Q}(X) + \mu_{Q}\abs{W_S Q}
\\ &= \Phi_{Q}(X) + \lambda_{Q}\abs{W_S Q} - \frac{\abs{W_S P}}{\abs{W_S Q}} c\cdot \abs{W_S Q} \quad\text{by definition of $\mu_Q$}
\\ &= \Phi_P(X) - \abs{W_S P}c \hspace{3.9cm}\text{by definition of $\lambda_Q$};
\end{split}
\end{equation}
which is independent of the choice of $Q\in [P]_{\cal F}$.

We define $\cal H':= \cal H \setminus [P]_{\cal F}$ as $\cal H$ with the $\cal F$-conjugates of $P$ removed. Because $P$ is maximal in $\cal H$, the subcollection $\cal H'$ again contains all $\cal F$-subconjugates of any $H\in \cal H'$.

From \eqref{eqX'Phi} we get that $X'$ is $\cal F$-stable away from $\cal H'$.
By the induction hypothesis we can therefore apply Lemma \ref{lemInduceFromASptoAFp} to $X'$ and the strictly smaller collection $\cal H'$. We get an element $\pi_{\cal H'} X'\in A(\cal F)_{(p)}$ which is an $\cal F$-stabilization of $X'$ over $\cal H'$, so $\pi_{\cal H'} X'$ satisfies
\begin{enumerate}
\item $\Phi_{Q}(\pi_{\cal H'} X')=\Phi_{Q}(X')$ and $c_{Q}(\pi_{\cal H'} X') = c_{Q}(X')$ for all $Q\not\in\cal H'$.
\item For all $Q\leq S$ we have
\[\sum_{[Q']_S \subseteq [Q]_{\cal F}} c_{Q'}(\pi_{\cal H'} X') = \sum_{[Q']_S \subseteq [Q]_{\cal F}} c_{Q'}(X').\]
\item For every $Q\leq S$:
\begin{align*}
\Phi_{Q}(\pi_{\cal H'} X') &= \frac{1}{\abs{[Q]_{\cal F}}}\sum_{Q'\in [Q]_{\cal F}} \Phi_{Q'}(X').
\end{align*}
\end{enumerate}
We claim that $\pi_{\cal H} X := \pi_{\cal H'} X'$ is an $\cal F$-stabilization of $X$ over $\cal H$ as well.

We immediately have that  $\Phi_{Q}(\pi_{\cal H'} X') = \Phi_{Q}(X')=\Phi_{Q}(X)$ and $c_{Q}(\pi_{\cal H'} X')=c_{Q}(X') = c_{Q}(X)$ for all $Q\not\in \cal H$, so property \ref{itemInducePreserve} is satisfied as well as property \ref{itemInduceCoord} for $Q\not\in\cal H$.
Since $c_Q(X') = c_{Q}(X)$ when $Q\not\sim_{\cal F} P$, we get for all $Q \in \cal H'$ that
\[\sum_{[Q']_S \subseteq [Q]_{\cal F}} c_{Q'}(\pi_{\cal H'} X') = \sum_{[Q']_S \subseteq [Q]_{\cal F}} c_{Q'}(X') = \sum_{[Q']_S \subseteq [Q]_{\cal F}} c_{Q'}(X).\]
Furthermore, since $P\not\in \cal H'$, we have $c_{P'}(\pi_{\cal H'} X') = c_{P'}(X')$ for $P' \sim_{\cal F} P$. Using \eqref{eqX'coord} we then get
\[\sum_{[P']_S \subseteq [P]_{\cal F}} c_{P'}(\pi_{\cal H'} X') = \sum_{[P']_S \subseteq [P]_{\cal F}} c_{P'}(X') = \sum_{[P']_S \subseteq [P]_{\cal F}} c_{P'}(X).\]
This proves that \ref{itemInduceCoord} is satisfied.
Since $c_Q(X') = c_{Q}(X)$ when $Q\not\sim_{\cal F} P$, we have $\Phi_Q (X') = \Phi_Q(X)$ for all $Q$ that are not $\cal F$-subconjugate to $P$. Consequently we have
\begin{align*}
\Phi_Q(\pi_{\cal H'} X') &= \frac{1}{\abs{[Q]_{\cal F}}} \sum_{Q'\in [Q]_{\cal F}} \Phi_{Q'}(X')
\\ &= \frac{1}{\abs{[Q]_{\cal F}}} \sum_{Q'\in [Q]_{\cal F}}  \Phi_{Q'}(X),
\end{align*}
when $Q$ is not $\cal F$-subconjugate to $P$.
By Lemma \ref{lemWeightedMeanOfCoordinates} every $P'\sim_{\cal F} P$ and $Q\leq S$ satisfies
\begin{equation}\label{eqWeightedMeanOfCoordinates}
\frac{1}{\abs{[Q]_{\cal F}}} \sum_{Q'\in [Q]_{\cal F}}  \Phi_{Q'}([S/P']) = \frac{\abs{\cal F(Q,P)}\cdot \abs{S}}{\abs P \cdot \abs{\cal F(Q,S)}}\in \Z_{(p)}.
\end{equation}
In the case where $Q$ is subconjugate to $P$ in $\cal F$, we can then use both \eqref{eqX'def} and \eqref{eqX'coord} to show that
\begin{align*}
  \Phi_Q (\pi_{\cal H'} X')
 \overset{\phantom{\eqref{eqX'def}}}={}& \frac{1}{\abs{[Q]_{\cal F}}} \sum_{Q'\in [Q]_{\cal F}} \Phi_{Q'}(X')
\\ \overset{\eqref{eqX'def}}={}& \frac{1}{\abs{[Q]_{\cal F}}} \sum_{Q'\in [Q]_{\cal F}}  \Phi_{Q'}(X) + \sum_{[P']_S\subseteq [P]_{\cal F}} \mu_{P'} \left(\frac{1}{\abs{[Q]_{\cal F}}} \sum_{Q'\in [Q]_{\cal F}} \Phi_{Q'}([S/P'])\right)
\\ \overset{\eqref{eqWeightedMeanOfCoordinates}}={}& \frac{1}{\abs{[Q]_{\cal F}}} \sum_{Q'\in [Q]_{\cal F}}  \Phi_{Q'}(X) + \sum_{[P']_S\subseteq [P]_{\cal F}} \mu_{P'}\cdot \frac{\abs{\cal F(Q,P)}\cdot \abs{S}}{\abs P \cdot \abs{\cal F(Q,S)}}
\\ \overset{\eqref{eqX'coord}}={}& \frac{1}{\abs{[Q]_{\cal F}}} \sum_{Q'\in [Q]_{\cal F}}  \Phi_{Q'}(X)+0;
\end{align*}
which proves that $\pi_{\cal H'} X'$ satisfies \ref{itemInduceWeightedMean}.
\end{proof}

\begin{mainthm}\label{thmStabilizationHom}
Let $\cal F$ be a saturated fusion system on a finite $p$-group $S$. We let $A(\cal F)_{(p)}$ denote the $p$-localized Burnside ring of $\cal F$ as a subring of the $p$-localized Burnside ring $A(S)_{(p)}$ for $S$.
For each $X\in A(S)_{(p)}$, there is a well defined $\cal F$-stable element $\pi(X)\in A(\cal F)_{(p)}$ determined by the fixed point formula
\begin{align*}
\Phi_{Q}(\pi(X)) ={}& \frac 1{\abs{[Q]_{\cal F}}}\sum_{Q'\in [Q]_{\cal F}} \Phi_{Q'}(X),
\end{align*}
which takes the average for each $\cal F$-conjugacy class $[Q]_{\cal F}$ of subgroups $Q\leq S$. The resulting map $\pi\colon A(S)_{(p)} \to A(\cal F)_{(p)}$ is a homomorphism of $A(\cal F)_{(p)}$-modules and restricts to the identity on $A(\cal F)_{(p)}$.
\end{mainthm}

\begin{rmk}
We will later show in Corollary \ref{corCharOnesidedStabilization} that the stabilization map $\pi\colon A(S)_{(p)} \to A(\cal F)_{(p)}$ in fact coincides with the map $A(S)_{(p)}\to A(S)_{(p)}$ induced by the $\cal F$-characteristic idempotent $\omega_{\cal F}$ as an element of the double Burnside ring $A(S,S)_{(p)}$. By constructing the map as we do here, instead of blindly applying the characteristic idempotent to define $\pi$, we gain a better understanding of the stabilization map: Of particular interest is the fixed point formula of Theorem \ref{thmStabilizationHom} itself as well as Remark \ref{rmkStabilizationCoeff}, which states that $\pi$ ``preserves orbits up to $\cal F$-conjugation''.
\end{rmk}

\begin{proof}[Proof of Theorem \ref{thmStabilizationHom}]
If we apply Lemma \ref{lemInduceFromASptoAFp} to $X$ and the collection $\cal H$ of all subgroups in $S$, we get a stable element $\pi_{\cal H}X\in A(\cal F)_{(p)}$ satisfying
\[\Phi_{Q}(\pi_{\cal H}X) = \frac 1{\abs{[Q]_{\cal F}}}\sum_{Q'\in [Q]_{\cal F}} \Phi_{Q'}(X)\]
as wanted. Hence $\pi(X) = \pi_{\cal H}X$ is an actual element of $A(\cal F)_{(p)}$.
If we apply $\pi$ to an element $X$ that is already $\cal F$-stable, then
\[\Phi_Q(\pi (X)) = \frac 1{\abs{[Q]_{\cal F}}}\sum_{Q'\in [Q]_{\cal F}} \Phi_{Q'}(X) = \frac 1{\abs{[Q]_{\cal F}}}\sum_{Q'\in [Q]_{\cal F}} \Phi_{Q}(X) = \Phi_{Q}(X),\]
so $\pi (X) = X$, and thus $\pi$ is the identity map when restricted to $A(\cal F)_{(p)}$.

If $X\in A(\cal F)_{(p)}$ and $Y\in A(S)_{(p)}$, then since the fixed point homomorphisms preserve products, we have
\begin{align*}
\Phi_Q(\pi(XY)) &= \frac 1{\abs{[Q]_{\cal F}}}\sum_{Q'\in [Q]_{\cal F}} \Phi_{Q'}(XY) = \frac 1{\abs{[Q]_{\cal F}}}\sum_{Q'\in [Q]_{\cal F}} \Phi_{Q'}(X)\Phi_{Q'}(Y)
\\ &= \Phi_{Q}(X)\cdot\frac 1{\abs{[Q]_{\cal F}}}\sum_{Q'\in [Q]_{\cal F}}\Phi_{Q'}(Y) = \Phi_{Q}(X) \cdot\Phi_Q(\pi(Y)).
\end{align*}
This shows that $\pi(XY) = X\cdot\pi(Y)$, and by a similar argument, $\pi$ preserves addition. Hence $\pi$ is a homomorphism of $A(\cal F)_{(p)}$-modules.
\end{proof}

\begin{rmk}\label{rmkStabilizationCoeff}
Since $\pi(X)=\pi_{\cal H} X$ is the $\cal F$-stabilization of $X$ over the collection $\cal H$ of all subgroups, property \ref{itemInduceCoord} of Definition \ref{dfnStabilization} states that
\[\sum_{[P']_S \subseteq [P]_{\cal F}} c_{P'}(\pi (X)) = \sum_{[P']_S \subseteq [P]_{\cal F}} c_{P'}(X) \quad \text{for all $P\leq S$.}\]
Hence $\pi$ replaces orbits of $X$ within each $\cal F$-conjugation class, but doesn't otherwise add or remove orbits from $X$.
This fact will be important for describing the action of the characteristic idempotent on bisets in Theorem \ref{thmCharIdemMultiplication} of section \ref{secCharIdem}.
\end{rmk}

We know that the transitive $S$ sets $[S/P]$ form a basis for $A(S)_{(p)}$. We now apply the projection $\pi\colon A(S)_{(p)} \to A(\cal F)_{(p)}$ to this basis, and we get a new basis for the $p$-localized Burnside ring $A(\cal F)_{(p)}$.
\begin{dfn}\label{dfnBetas}
For each $P\leq S$ we define $\beta_{P}\in A(\cal F)_{(p)}$ to be the element $\beta_P:= \pi([S/P])$.
\end{dfn}

\begin{rmk}\label{rmkBetas}
By Theorem \ref{thmStabilizationHom} and Lemma \ref{lemWeightedMeanOfCoordinates} the element $\beta_{P}$ satisfies
\[\Phi_{Q}(\beta_{P}) = \frac 1{\abs{[Q]_{\cal F}}}\sum_{Q'\in [Q]_{\cal F}}\Phi_{Q'}([S/P]) = \frac{\abs{\cal F(Q,P)}\cdot \abs{S}}{\abs P \cdot \abs{\cal F(Q,S)}} \in \Z_{(p)},\]
where the last expression, and thus $\beta_P$, only depends on $P$ up to conjugation in $\cal F$.
\end{rmk}

\begin{prop}\label{thmPLocalBurnsideBasis}
Let $\beta_{P}\in A(\cal F)_{(p)}$ be defined as above.
The set of these elements $\{\beta_P\mid P\leq S\text{ up to $\cal F$-conjugation}\}$ is a $\Z_{(p)}$-basis for $A(\cal F)_{(p)}$.
\end{prop}

\begin{proof}
Because the transitive $S$-sets $[S/P]$ for $P\leq S$ generate $A(S)_{(p)}$, and since $\pi$ is surjective, the elements $\beta_P$ must generate all of $A(\cal F)_{(p)}$.

We now order the $\cal F$-conjugacy classes $[P]_{\cal F}$ according to decreasing order of $P$, and the
mark homomorphism $\Phi\colon \Span\{\beta_{P}\} \to \free{\cal F}_{(p)}$ is then represented by a matrix $M$ with entries
\[M_{Q,P} = \frac{\abs{\cal F(Q,P)}\cdot \abs{S}}{\abs P \cdot \abs{\cal F(Q,S)}}.\]
If $Q$ is not $\cal F$-subconjugate to $P$, then $M_{Q,P}=0$; so $M$ is a lower triangular matrix with diagonal entries
\[M_{P,P} = \frac{\abs{\cal F(P,P)}\cdot \abs{S}}{\abs P \cdot \abs{\cal F(P,S)}}\neq 0.\]
Since all diagonal entries are non-zero, we conclude that the $\beta_P$ are linearly independent over $\Z_{(p)}$.
\end{proof}

The mark homomorphism $\Phi\colon A(S)_{(p)}\to \free{S}_{(p)}$ embeds the Burnside ring of $S$ into its ghost ring, and since we know the value of $\Phi_Q(\beta_P)$ from Remark \ref{rmkBetas}, we know the image of $\beta_P$ inside $\free{S}_{(p)}$. We might then wonder whether we can pull back our knowledge from $\free S_{(p)}$ to $A(S)_{(p)}$ and write $\beta_P$ explicitly as a linear combination of transitive $S$-sets.

In \cite{Gluck}, David Gluck gives a method on how to do exactly this. Because $A(S)_{(p)}$ embeds in the ghost ring $\free{S}_{(p)}$ as a subring of finite index, if we take the tensor product with $\Q$, we get an isomorphism $\Phi\colon A(S)\ox \Q \xto{\cong} \free S\ox \Q.$
Let $e_Q:= (0,\dotsc,0,1,0,\dotsc,0)$ be the standard basis element of $\free{S}\ox \Q$ corresponding to the subgroup $Q\leq S$.

According to the one Proposition of \cite{Gluck} the inverse $\Phi^{-1}\colon\free S \ox \Q \to A(S)\ox \Q$ is then given by
\begin{equation}\label{eqPhiInverse}
\Phi^{-1}(e_Q) =\frac1{\abs{N_S Q}} \sum_{R\leq Q} \mu(R,Q)\cdot \abs R\cdot [S/R],
\end{equation}
where $\mu$ is the Möbius-function for the poset of subgroups in $S$.

Since we know the image $\Phi(\beta_P)$, we can apply the isomorphism above to get an expression for $\beta_P$ inside $A(S)\ox \Q$; and because $A(S)_{(p)}$ is embedded in $A(S)\ox \Q$, the expression holds in $A(S)_{(p)}$ as well.
\begin{prop}\label{propPLocalBasisDecomp}
For each $P\leq S$, the element $\beta_P\in A(\cal F)_{(p)}$ is given by the following expression when written as a $\Z_{(p)}$-linear combination of transitive $S$-sets:
\begin{align*}
\beta_P &= \sum_{[R]_S} \frac {1}{\Phi_R([S/R])} \left(\sum_{R\leq Q\leq S} \Phi_Q(\beta_P)\cdot \mu(R,Q) \right)[S/R]
\\ &= \sum_{[R]_S} \frac {\abs R\cdot \abs S}{\abs{N_S R}\cdot\abs P} \left(\sum_{R\leq Q\leq S} \frac{\abs{\cal F(Q,P)}}{\abs{\cal F(Q,S)}}\cdot \mu(R,Q) \right)[S/R].
\end{align*}
In particular $\beta_P$ contains no copies of $[S/R]$ unless $R$ is $\cal F$-subconjugate to $P$.
\end{prop}

\begin{proof}
By definition of the mark homomorphism $\Phi$ and the idempotents $e_Q$ we have
\begin{align*}
\Phi(\beta_P) &= \sum_{[Q]_S} \Phi_Q(\beta_P)\cdot e_Q
= \sum_{Q \leq S} \frac{\abs{N_S Q}}{\abs S}\cdot \Phi_Q(\beta_P)\cdot e_Q.
\end{align*}
We then apply the formula \eqref{eqPhiInverse} for the inverse of $\Phi$ and get
\begin{align*}
\beta_P &= \Phi^{-1}\left(\sum_{Q \leq S} \frac{\abs{N_S Q}}{\abs S}\cdot \Phi_Q(\beta_P)\cdot e_Q\right)
\\ &= \sum_{Q \leq S} \frac{\abs{N_S Q}}{\abs S}\cdot \Phi_Q(\beta_P)\cdot \frac1{\abs{N_S Q}} \left(\sum_{R\leq Q} \mu(R,Q)\cdot \abs R\cdot [S/R]\right)
\\ &= \sum_{R\leq S} \frac {\abs R}{\abs S} \left(\sum_{R\leq Q\leq S} \Phi_Q(\beta_P)\cdot \mu(R,Q) \right)[S/R]
\\ &= \sum_{[R]_S} \frac {\abs R}{\abs{N_S R}} \left(\sum_{R\leq Q\leq S} \Phi_Q(\beta_P)\cdot \mu(R,Q) \right)[S/R]
\\ &= \sum_{[R]_S} \frac {\abs R\cdot \abs S}{\abs{N_S R}\cdot\abs P} \left(\sum_{R\leq Q\leq S} \frac{\abs{\cal F(Q,P)}}{\abs{\cal F(Q,S)}}\cdot \mu(R,Q) \right)[S/R].
\end{align*}
If $R$ is not $\cal F$-subconjugate to $P$, then $\abs{\cal F(Q,P)}=0$ for all $R\leq Q\leq S$, and hence the coefficient of $[S/R]$ above becomes zero.
\end{proof}

\subsection{Equivalent Burnside rings}\label{secEquivBurnsideRings}
In this section we compare the ring $A(\cal F)_{(p)}$, with the $\beta_P$-basis, to other Burnside rings related to the saturated fusion system $\cal F$. First we consider the case where $\cal F$ is realized by a group $G$: We see that $A(\cal F)_{(p)}$ is isomorphic to the ring $A(G;p)_{(p)}$ generated by $G$-sets $[G/P]$ where $P\leq G$ is a $p$-group, and the basis element $\beta_P$ corresponds to the transitive $G$-set $[G/P]$ up to normalization with a common invertible element in $A(G;p)_{(p)}$. After that, we consider the Burnside ring $A^{cent}(\cal F)$ introduced by Antonio Diaz and Assaf Libman in \cite{DiazLibman}, which is defined using only the centric subgroups of $\cal F$: We show that after $p$-localization $A^{cent}(\cal F)_{(p)}$ is isomorphic to the ``centric part'' of $A(\cal F)_{(p)}$, again with the basis elements corresponding to each other in a suitable way.
Both of these isomorphisms are originally due to Diaz-Libman in \cite{DiazLibman:Segal} as Example 3.9 and Theorem A, respectively. New in this section is the fact that the bases of the rings correspond as well.

\begin{prop}\label{propSameBurnsideGroup}
Suppose that $S$ is a Sylow $p$-subgroup of $G$, and let $\cal F:= \cal F_S(G)$. Define $A(G;p)$ to be the subring of $A(G)$ where all isotropy subgroups are $p$-groups.

Then the transitive $G$-set $[G/S]$ is invertible in $A(G;p)_{(p)}$, and we get an isomorphism of rings $A(\cal F)_{(p)} \cong A(G;p)_{(p)}$ by
\[\beta_P \mapsto \frac {[G/P]}{[G/S]}.\]
\end{prop}
Note that dividing by $[G/S]$ only makes sense in $A(G;p)_{(p)}$ and not in $A(G)_{(p)}$. The rings $A(G;p)_{(p)}$ and $A(G)_{(p)}$ have different $1$-elements, and $[G/S]$ is not invertible in $A(G)_{(p)}$.

The isomorphism above is in a way the best we could hope for, since the basis element $\beta_P$ only depends on the fusion data in $\cal F_S(G)$, while the transitive $G$-set $[G/P]$ depends on the actual group $G$. If we replace $G$ with a product $G'=G\x H$ where $H$ is a $p'$-group, then the fusion system $\cal F_S(G')$ and $\beta_P$ are the same for $G'$ as for $G$, but the transitive set $[G'/P]$ has increased in size by a factor $\abs H$. However, as we see in the proof below, the quotient $\frac {[G/P]}{[G/S]}$ depends only on the fusion system and not on $G$.

\begin{proof}
We first show that $[G/S]$ is invertible in $A(G)_{(p)}$. For every $Q\leq S$ that is fully $\cal F$-normalized, we have
\[\Phi_Q([G/S]) = \frac{\abs{N_G(Q,S)}}{\abs S} = \frac{\abs{N_G Q} \cdot \abs{\{Q'\leq S|Q'\sim_\cal F Q\}}}{\abs S}.\]
By Lemma \ref{lemNumberOfConjugates}, we have $\abs{\{Q'\leq S|Q'\sim_\cal F Q\}}=\frac{\abs S}{\abs{N_S Q}}\cdot k$ with $p\nmid k$. We thus get
\[\Phi_Q([G/S]) = \frac {\abs{N_G Q}}{\abs{N_S Q}}\cdot k,\]
which is invertible in $\Z_{(p)}$ since $Q$ is fully $\cal F$-normalized.

If $H\leq G$ is not a $p$-group, then $\Phi_H(X) = 0$ for all $X\in A(G;p)_{(p)}$. We also know that every $p$-subgroup of $G$ is conjugate to a subgroup of $S$ by Sylow's theorems, and therefore the mark homomorphism for $A(G)_{(p)}$ restricts to an injection
\[\Phi \colon A(G;p)_{(p)} \to \prod_{[Q]_{\cal F}} \Z_{(p)} = \free{\cal F}_{(p)},\]
and $A(G;p)_{(p)}$ has finite index in $\free{\cal F}_{(p)}$ for rank reasons.

Because $\Phi_Q([G/S])$ is invertible in $\Z_{(p)}$, $[G/S]$ is invertible in the ghost ring $\free{\cal F}_{(p)}$. It follows that multiplication with $[G/S]$ is a bijection $\free{\cal F}_{(p)} \to \free{\cal F}_{(p)}$, which sends $A(G;p)_{(p)}$ into itself. Since $A(G;p)_{(p)}$ has finite index in $\free{\cal F}_{(p)}$, multiplication with $[G/S]$ must then also be a bijection of $A(G;p)_{(p)}$ to itself, and because this bijection hits $[G/S]$, the $1$-element of $\free{\cal F}_{(p)}$ must lie in $A(G;p)_{(p)}$ and be the $1$-element of this subring. In addition, $[G/S]$ must be invertible in $A(G;p)_{(p)}$.

It thus makes sense to consider the elements $\frac{[G/P]}{[G/S]}$ in $A(G;p)_{(p)}$ for $P\leq S$, and we calculate
\[\Phi_Q\left(\frac{[G/P]}{[G/S]}\right) = \frac{\abs{N_G(Q,P)}\cdot\abs S}{\abs P \cdot \abs{N_G(Q,S)}} = \frac{\abs{\cal F(Q,P)}\cdot\abs S}{\abs P \cdot \abs{\cal F(Q,S)}} = \Phi_Q(\beta_P).\]
It follows that $\frac{[G/P]}{[G/S]}= \beta_P$ as elements of $\free{\cal F}_{(p)}$, giving the isomorphism $A(G;p)_{(p)} \cong A(\cal F)_{(p)}$.
\end{proof}

The Burnside ring defined by Diaz-Libman in \cite{DiazLibman} for a saturated fusion system $\cal F$, is constructed in terms of an orbit category over the $\cal F$-centric subgroups of $S$. A subgroup $P\leq S$ is $\cal F$-centric if all $\cal F$-conjugates $P'\sim_{\cal F} P$ are self-centralizing, i.e., $C_S(P')\leq P'$.
We denote the Diaz-Libman Burnside ring by $A^{cent}(\cal F)$, and it comes equipped with an additive basis $\xi_P$ indexed by the $\cal F$-conjugacy classes of $\cal F$-centric subgroups.
As shown in \cite{DiazLibman} there is also an injective homomorphism of marks
\[\Phi^{cent} \colon  A^{cent}(\cal F) \to \prod_{\substack{[P]_{\cal F}\\ \text{$P$ is $\cal F$-centric}}} \Z\]
with finite cokernel, and on basis elements $\Phi^{cent}$ is given by
\[\Phi^{cent}_Q(\xi_P) = \frac{\abs{Z(Q)}\cdot \abs{\cal F(Q,P)}}{\abs P}.\]

\begin{prop}\label{propSameBurnsideCent}
Let $\cal F$ be a saturated fusion system on a $p$-group $S$, and write\linebreak $N\leq A(\cal F)_{(p)}$ for the $\Z_{(p)}$-submodule generated by $\beta_P$ for non-$\cal F$-centric $P$.
Then $N$ is an ideal in the Burnside ring $A(\cal F)_{(p)}$, and there is a ring isomorphism $A(\cal F)_{(p)}/N\cong A^{cent}(\cal F)_{(p)}$ with the Burnside ring of Diaz-Libman. The basis element $\xi_S$ is invertible in $A^{cent}(\cal F)_{(p)}$, and the isomorphism is given by
\[\bar{\beta_P} \mapsto \frac{\xi_P}{\xi_S}\]
for $\cal F$-centric $P\leq S$.
\end{prop}

\begin{proof}
If $P$ is $\cal F$-centric, then any subgroup containing $P$ is $\cal F$-centric as well, hence the collection of \emph{non}-$\cal F$-centric subgroups is closed under $\cal F$-conjugation and taking subgroups. By the double coset formula \eqref{eqSingleBurnsideDoubleCoset} for $A(S)_{(p)}$, the $\Z_{(p)}$-submodule generated by the elements $[S/P]$ with $P$ non-$\cal F$-centric is an ideal in $A(S)_{(p)}$. Let us denote this ideal $M\leq A(S)_{(p)}$.

The stabilization map $\pi\colon A(S)_{(p)}\to A(\cal F)_{(p)}$ is a homomorphism of $A(\cal F)_{(p)}$-modules, so the image $N=\pi(M)$ is an ideal of $A(\cal F)_{(p)}$, and at the same time $N$ is the $\Z_{(p)}$-submodule generated by the elements $\pi([S/P])=\beta_P$ where $P$ is non-$\cal F$-centric.
By Proposition \ref{thmPLocalBurnsideBasis}, we have $\Phi_Q(\beta_P)=0$ whenever $Q$ is $\cal F$-centric and $P$ is not. Hence the homomorphism
\[A(\cal F)_{(p)} \xto\Phi \prod_{[P]_{\cal F}} \Z_{(p)} \to \prod_{\substack{[P]_{\cal F}\\ \text{$P$ is $\cal F$-centric}}} \Z_{(p)}\]
sends $N$ to $0$, and therefore induces a ring homomorphism \[\Phi\colon A(\cal F)_{(p)}/N \to \prod_{\substack{[P]_{\cal F}\\ \text{$P$ is $\cal F$-centric}}} \Z_{(p)}.\]
Let $\bar{\beta_P}$ denote the equivalence class of $\beta_P$ in $A(\cal F)_{(p)}/N$ when $P$ is $\cal F$-centric. The quotient ring $A(\cal F)_{(p)}/N$ then has a basis consisting of $\bar{\beta_P}$ for each $\cal F$-centric $P$ up to $\cal F$-conjugation.

The rest of this proof follows the same lines as the proof of Proposition \ref{propSameBurnsideGroup}: For the basis element $\xi_S$ of $A^{cent}(\cal F)_{(p)}$ the image under the mark homomorphism has the form $\Phi^{cent}_Q(\xi_S) = \frac {\abs {Z(Q)}}{\abs S} \cdot \abs*{\cal F(Q,S)}$, which by Lemma \ref{lemNumberOfConjugates} is invertible in $\Z_{(p)}$. Hence $\xi_S$ is invertible in the ghost ring
\[\prod_{\substack{[P]_{\cal F}\\ \text{$P$ is $\cal F$-centric}}} \Z_{(p)},\]
and since $\Phi^{cent}$ has finite cokernel, it follows that multiplying by $\xi_S$ gives a bijection of $A^{cent}(\cal F)_{(p)}$ to itself, the $1$-element of the ghost ring is a $1$-element in $A^{cent}(\cal F)_{(p)}$ as well, and $\xi_S$ is invertible in $A^{cent}(\cal F)_{(p)}$. It therefore makes sense to form the fractions $\frac{\xi_P}{\xi_S}$ in $A^{cent}(\cal F)_{(p)}$. Applying the fixed point homomorphisms to these fractions, we then get
 \[\Phi^{cent}_Q\left(\frac{\xi_P}{\xi_S}\right) = \frac {\abs{Z(Q)}\cdot \abs{\cal F(Q,P)}\cdot \abs{S}}{\abs P\cdot \abs{\cal F(Q,S)}\cdot \abs{Z(Q)}} = \frac{\abs{\cal F(Q,P)}\cdot \abs S}{\abs P \cdot \abs{\cal F(Q,S)}} = \Phi_Q(\beta_P)\] for all $\cal F$-centric subgroup $Q,P\leq S$. This shows that the ring homomorphism
\[\Phi\colon A(\cal F)_{(p)}/N \to \prod_{\substack{[P]_{\cal F}\\ \text{$P$ is $\cal F$-centric}}} \Z_{(p)}\]
sends $\bar{\beta_P}$ to $\Phi^{cent}(\frac {\xi_P}{\xi_S})$, which proves that $\Phi$ is injective on $A(\cal F)_{(p)}/N$ and that $\beta_P\mapsto \frac {\xi_P}{\xi_S}$ gives a ring isomorphism $A(\cal F)_{(p)}/N\cong A^{cent}(\cal F)_{(p)}$.
\end{proof}

\section{The characteristic idempotent}\label{secCharIdem}
In this section we make use of the stabilization homomorphism of Theorem \ref{thmStabilizationHom} to give new results on the characteristic idempotent for a saturated fusion system. These idempotents were shown by Ragnarsson and Stancu to classify the saturated fusion systems on a given $p$-group.
In section \ref{secConstructCharIdem} we calculate the fixed points and orbit decomposition of the characteristic idempotent $\omega_{\cal F}$ for a saturated fusion system $\cal F$ on $S$ by applying the results of section \ref{secStabilization} to the product fusion system $\cal F\x\cal F$ on $S\x S$.
The results of these calculations form the content of Theorem \ref{thmCharIdemStructure}.
In section \ref{secCharIdemAction} we discuss multiplication $X\mapsto \omega_{\cal F}\circ X$ with the characteristic idempotent -- for elements $X$ of the double Burnside ring of $S$ and more generally when $X$ is just some finite set with an $S$-action. Theorem \ref{thmCharIdemMultiplication} describes the multiplication with $\omega_{\cal F}$ in terms of the homomorphism of marks.

\subsection{Bisets and Burnside modules}\label{secDoubleBurnside}
For finite groups $G$ and $H$, a $(G,H)$-biset is a set with both a left $H$-action and a right $G$-action, and such that the two actions commute. A $(G,H)$-biset $X$ gives rise to a $(H\x G)$-set by defining $(h,g).x := hxg^{-1}$, and vice versa. The transitive $(G,H)$-bisets have the form $[(H\x G)/D]$ for subgroups $D\leq H\x G$.
The isomorphism classes of finite $(G,H)$-bisets form a monoid, and the Grothendieck group $A(G,H)$ is called the \emph{Burnside module} of $G$ and $H$. Additively $A(G,H)$ is isomorphic to $A(H\x G)$ and we have a basis consisting of the transitive bisets $[(H\x G)/D]$ where $D\leq H\x G$ is determined up to $(H\x G)$-conjugation.

The multiplication for the Burnside modules is different from the non-biset Burnside rings. We have multiplication/composition maps $\circ \colon A(H,K) \x A(G,H) \to A(G,K)$, defined for every $(G,H)$-biset $X$ and $(H,K)$-biset $Y$ as
\[Y\circ X := Y\x_H X = Y\x X/ \sim\]
where $(yh,x) \sim (y,hx)$ for all $y\in Y$, $x\in X$ and $h\in H$.
The composition is associative, and for each finite group $G$ the group itself considered as a biset is an identity element in $A(G,G)$ with respect to the composition. The composition makes $A(G,G)$ a ring, we call this the \emph{double Burnside ring} of $G$.
On transitive bisets, the composition is given by a double coset formula
\begin{equation}\label{eqDoubleBurnsideDoubleCoset}
[(K\x H)/D] \circ [(H\x G)/C] \sum_{\bar x \in \pi_2 D \backslash H / \pi_1 C} [(K\x G) / (D * \lc {(x,1)} C)]
\end{equation}
where $\lc{(x,1)} C$ denotes the conjugate of $C$ by the element $(x,1)\in H\x G$, and the subgroup $B* A$ is defined as $\{(k,g)\in K\x G \mid \exists h\in H \colon (k,h)\in B, (h,g)\in A\}$ for subgroups $B\leq K\x H$ and $A\leq H\x G$.

Given a homomorphism $\ph\colon U\to H$ with $U\leq G$, the graph $\Delta(U,\ph)=\{(\ph u, u) \mid u\in U\}$ is a subgroup of $H\x G$. We introduce the notation $[U,\ph]_G^H$ as a shorthand for the biset $[(H\x G)/\Delta(U,\ph)]$, and if the groups $G,H$ are clear from context, we just write $[U,\ph]$.
The bisets $[U,\ph]$ generate the $(G,H)$-bisets that have a free left $H$-action. For these basis elements, \eqref{eqDoubleBurnsideDoubleCoset} takes the form
\begin{equation}\label{eqDoubleBurnsideDoubleCosetFree}
[T,\psi]_H^K \circ [U,\ph]_G^H = \sum_{\bar x \in T\backslash H / \ph U} [\ph^{-1}(T^x)\cap U, \psi c_x\ph]_G^K.
\end{equation}
%If either $T=H$ or $\ph(U)=H$, the formula simplifies to $[T,\psi]_H^K \circ [U,\ph]_G^H = [\ph^{-1}(T)\cap U,\ph\psi]_G^K$.

From the isomorphism $A(G,H)\cong A(H\x G)$ of additive groups, the Burnside modules inherit fixed point homomorphisms $\Phi_C\colon A(G,H) \to \Z$ for each $(H\x G)$-conjugacy class of subgroups $C\leq H\x G$.
Note however that the fixed point homomorphisms for $A(G,G)$ are \emph{not} ring homomorphisms -- they are only homomorphisms of abelian groups.

Given any $(G,H)$-biset $X$, we can swap the actions to get an $(H,G)$-biset $X^\op$ with $g.x^\op.h := h^{-1}.x.g^{-1}$, which extends to a group isomorphism $(-)^\op \colon A(G,H) \to A(H,G)$. We clearly have $[(H\x G)/D]^\op = [(G\x H)/D^\op]$ and $\Phi_C(X^\op)=\Phi_{C^\op}(X)$, where $C^\op,D^\op$ are the subgroups $C,D$ with the coordinates swapped.
Any element of the double Burnside ring $X\in A(G,G)$ that satisfies $X^\op=X$ is called \emph{symmetric}.

\subsection{The structure of the characteristic idempotent}\label{secConstructCharIdem}
Let $\cal F$ be a fusion system on a $p$-group $S$.
We then say that an element of the $p$-localized double Burnside ring $A(S,S)_{(p)}$ is $\cal F$-characteristic if it satisfies the Linckelmann-Webb properties:
The element is $\cal F$-generated (see \ref{dfnFgenerated}), it is $\cal F$-stable (see \ref{dfnFstable}), and finally there is a $p'$-condition for the number of elements (see \ref{dfnFchar}).

K. Ragnarsson showed in \citelist{\cite{Ragnarsson}} that for every saturated fusion system $\cal F$ on a $p$-group $S$, there is a unique idempotent $\omega_{\cal F}\in A(S,S)_{(p)}$ that is $\cal F$-characteristic, and \cite{RagnarssonStancu} shows how $\cal F$ can be reconstructed from $\omega_{\cal F}$ (or any $\cal F$-characteristic element).
The goal of this section is to prove Theorem \ref{thmCharIdemStructure} giving the value of the mark homomorphism on $\omega_{\cal F}$ and the decomposition of $\omega_{\cal F}$ into $(S,S)$-biset orbits. To prove the theorem, we consider the basis element $\beta_{\Delta(S)}$ for the diagonal subgroup $\Delta(S)\leq S\x S$ with respect to the product fusion system $\cal F\x\cal F$ on $S\x S$. We show that $\beta_{\Delta(S)}$ is in fact the characteristic idempotent for $\cal F$ which allows us to apply the formulas of Remark \ref{rmkBetas} and Proposition \ref{propPLocalBasisDecomp} to $\omega_{\cal F}$ directly.

\begin{dfn}\label{dfnFgenerated}
Let $\cal F$ be a fusion system on a $p$-group $S$.
An element $X\in A(S,S)$ is then said to be \emph{$\cal F$-generated} if $X$ is expressed solely in terms of basis elements $[P,\ph]$ where $\ph\colon P\to S$ is a morphism of $\cal F$. The $\cal F$-generated elements form a unital subring $A_{\cal F}(S,S)$ of the double Burnside ring.
\end{dfn}

\begin{rmk}
Since $[P,\ph]^\op=[\ph P, \ph^{-1}]$ for all $\ph\in \cal F(P,S)$, the ring $A_{\cal F}(S,S)$ of $\cal F$-generated elements is stable with respect to the reflection $(-)^\op$.
\end{rmk}

\begin{rmk}
Any subgroup of a graph $\Delta(P,\ph)$ with $\ph\in \cal F(P,S)$ has the form $\Delta(R,\ph|_R)$ for some subgroup $R\leq P$. By \eqref{eqPhiOnBasis} we thus have $\Phi_D([P,\ph])=0$ unless $D$ is the graph of a morphism in $\cal F$. An element $X\in A(S,S)_{(p)}$ is therefore $\cal F$-generated if and only if $\Phi_D([P,\ph])=0$ for all subgroups $D\leq S\x S$ that are not graphs from $\cal F$.
\end{rmk}

\begin{dfn}\label{dfnFstable}
For the Burnside ring of a group $A(S)$ we defined by \eqref{charFstable} what it means for an $S$-set to be $\cal F$-stable. With bisets we now have both a left and a right action, hence we get two notions of stability:

Let $\cal F_1,\cal F_2$ be fusion systems on $p$-groups $S_1,S_2$ respectively.
Any $X\in A(S_1,S_2)_{(p)}$ is said to be \emph{right $\cal F_1$-stable} if it satisfies
\begin{equation}\label{charRightFstable}
\parbox[c]{.9\textwidth}{\emph{$X\circ [P,\ph]_P^{S_1} = X\circ [P,id]_P^{S_1}$ inside $A(P,S_2)_{(p)}$, for all $P\leq S_1$ and $\ph\colon P\to S_1$ in $\cal F_1$.}}
\end{equation}
Similarly $X\in A(S_1,S_2)_{(p)}$ is \emph{left $\cal F_2$-stable} if is satisfies
\begin{equation}\label{charLeftFstable}
\parbox[c]{.9\textwidth}{\emph{$[\ph P,\ph^{-1}]_{S_2}^P\circ X = [P,id]_{S_2}^P \circ X$ inside $A(S_1,P)_{(p)}$, for all $P\leq S_2$ and $\ph\colon P\to S_2$ in $\cal F_2$.}}
\end{equation}
For the double Burnside ring $A(S,S)_{(p)}$, any element that is both left and right $\cal F$-stable is said to be \emph{fully $\cal F$-stable} or just \emph{$\cal F$-stable}.
\end{dfn}

\begin{rmk}
Because $([P,\ph])^\op=[\ph P,\ph^{-1}]$ when $\ph$ is injective, we clearly have that $X$ is right $\cal F$-stable if and only if $X^\op$ is left $\cal F$-stable.
When $X$ is a right $\cal F$-stable biset, it is also clear that any composition $Y\circ X$ is right $\cal F$-stable as well by associativity of the composition $\circ$, and similarly for left-stable bisets.
\end{rmk}

As with $\cal F$-stability in $A(S)_{(p)}$, we can characterize left and right stability in terms of the homomorphism of marks for the double Burnside ring.
\begin{lem}\label{lemStableBiset}
Let $\cal F_1,\cal F_2$ be fusion systems on $p$-groups $S_1,S_2$ respectively. The following are then equivalent for all $X\in A(S_1,S_2)_{(p)}$:
\begin{enumerate}
\item\label{itemDoubleStable} $X$ is both right $\cal F_1$-stable and left $\cal F_2$-stable.
\item\label{itemProductStable} $X$ considered as an element of $A(S_2\x S_1)_{(p)}$ is $(\cal F_2\x \cal F_1)$-stable.
\item\label{itemProductPhi} $\Phi_D(X) = \Phi_{D'}(X)$ for all subgroups $D,D'\leq S_2\x S_1$ that are $(\cal F_2\x \cal F_1)$-conjugate.
\end{enumerate}
The analogous statements for right and left stability follow if we let $\cal F_1$ or $\cal F_2$ be trivial fusion systems.
\end{lem}
For the purposes of this paper it would be sufficient to state Lemma \ref{lemStableBiset} and later results only for bisets where both actions are free, in which case the proof of Lemma \ref{lemStableBiset} would be easier. However, all the later proofs are nearly identical in the bifree and non-free cases, so for completeness sake we include the general statements -- though the following proof becomes harder.

\begin{proof}
The equivalence of \ref{itemProductStable} and \ref{itemProductPhi} is just the characterization of stability in Burnside rings (see page \pageref{itemPhiBurnsideEq}).

Suppose that $X\in A(S_1,S_2)_{(p)}$ is both right $\cal F_1$-stable and left $\cal F_2$-stable. Let the map $\ph\in \Hom_{\cal F_2\x \cal F_1}(D,S_2\x S_1)$ be any homomorphism in the product fusion system, and let the ring homomorphism $\ph^* \colon A(S_2\x S_1)_{(p)}\to A(D)_{(p)}$ be the restriction along $\ph$. For subgroups $D\leq C\leq S_2\x S_1$ we also let $incl_D^C$ denote the inclusion of $D$ in $C$. We then wish to show that $\ph^*(X) = (incl_D^{S_2\x S_1})^*(X)$. Define $D_i$ to be the projection of $D$ to the group $S_i$, then by definition of the product fusion system $\ph$ has the form $(\ph_2\x \ph_1)|_D$ for suitable morphisms $\ph_i\in \cal F_i(D_i,S_i)$. The restriction homomorphism $\ph^*$ thus decomposes as
\[\ph^*\colon A(S_2\x S_1)_{(p)} \xto{(\ph_2\x \ph_1)^*} A(D_2\x D_1)_{(p)} \xto{(incl_D^{D_2\x D_1})^*} A(D)_{(p)}.\]
On $(S_1,S_2)$-bisets the composition
\[[\ph_2D_2, \ph_2^{-1}]_{S_2}^{D_2} \circ X \circ [D_1,\ph_1]_{D_1}^{S_1}\] is exactly the same as the restriction $(\ph_2\x \ph_1)^*$ of $(S_2\x S_1)$-sets, and by the assumed stability of $X$ we therefore get
\begin{align*}
(\ph_2\x \ph_1)^*(X) &= [\ph_2D_2, \ph_2^{-1}]_{S_2}^{D_2} \circ X \circ [D_1,\ph_1]_{D_1}^{S_1}
\\ &= [D_2, id]_{S_2}^{D_2} \circ X \circ [D_1,id]_{D_1}^{S_1} = (incl_{D_2\x D_1}^{S_2\x S_1})^*(X).
\end{align*}
Restricting further to $D$, we then have $\ph^*(X) = (incl_D^{S_2\x S_1})^*(X)$ as claimed.

Suppose conversely that $X$ is $\cal F_2\x \cal F_1$-stable. Then in particular we assume that\linebreak $(id\x\ph)^*(X) = (incl_{S_2\x P}^{S_2\x S_1})^*(X)$ for all maps $\ph\in \cal F_1(P,S_1)$, hence we have
\[X\circ [P,\ph]_P^{S_1} = (id\x \ph)^*(X) = (incl_{S_2\x P}^{S_2\x S_1})^*(X) = X\circ [P,id]_P^{S_1}\] so $X$ is right $\cal F_1$-stable. Similarly we get that $X$ is left $\cal F_2$-stable as well.
\end{proof}

Let $A^{\lhd}(S,S)_{(p)}$ be the subring of the double Burnside ring generated by left-free bisets, i.e., the subring with basis elements $[P,\ph]$ where $\ph\colon P\to S$ is any group homomorphism. In particular, every $\cal F$-generated element lies in $A^\lhd(S,S)_{(p)}$. We then define an augmentation map $\e(X) := \frac{\abs X}{\abs S}$ for any biset $X\in A^\lhd(S,S)_{(p)}$. Since $\e(X\circ Y) = \frac {\abs{X\x_S Y}}{\abs S} = \frac {\abs X \abs Y}{\abs S^2} = \e(X)\e(Y)$, we get a ring homomorphism $\e\colon A^{\lhd}(S,S)_{(p)} \to \Z_{(p)}$.

\begin{dfn}\label{dfnFchar} Let $\cal F$ be a fusion system on a $p$-group $S$.
An element $X\in A^\lhd(S,S)_{(p)}$ is said to be \emph{right/left/fully $\cal F$-characteristic} if:
\begin{enumerate}
\item\label{itemCharFgen} $X$ is $\cal F$-generated.
\item\label{itemCharFstable} $X$ is right/left/fully $\cal F$-stable, respectively.
\item\label{itemCharNonDegen} $\e(X)$ is invertible in $\Z_{(p)}$.
\end{enumerate}
A fully $\cal F$-characteristic element is also just called \emph{$\cal F$-characteristic}.
\end{dfn}

To prove Theorem \ref{thmCharIdemStructure} we consider the saturated fusion system $\cal F\x \cal F_S$ on $S\x S$, where $\cal F_S:=\cal F_S(S)$ is the trivial fusion system on $S$. For this product fusion system we then apply the stabilization map of Theorem \ref{thmStabilizationHom} to $[S,id]$ and get $\beta_{\Delta(S)}\in A(\cal F\x \cal F_S)_{(p)}$. By construction $\beta_{\Delta(S)}$ is only left $\cal F$-stable, but the fixed point calculation in Lemma \ref{lemCharIdemFixedPoints} will show that $\beta_{\Delta(S)}$ is right stable as well. Finally, using Remark \ref{rmkStabilizationCoeff}, we will show that $\beta_{\Delta(S)}$ is the characteristic idempotent for $\cal F$.

Alternatively, we could in theory stabilize $[S,id]$ with respect to $\cal F\x \cal F$, in order to immediately get a fully $\cal F$-stable element. The fixed point formulas imply that this would give us exactly the same element $\beta_{\Delta(S)}$ as before. However, when stabilizing with respect to a larger fusion system, Remark \ref{rmkStabilizationCoeff} yields less information about the orbits of the stabilized element, hence it would be harder to show that $\beta_{\Delta(S)}$ is idempotent. This is why we use the asymmetric approach with $\cal F\x \cal F_S$.

\begin{lem}\label{lemCharIdemFixedPoints}
Let $\cal F$ be a saturated fusion system on a $p$-group $S$, and let $\cal F_S:=\cal F_S(S)$ denote the trivial fusion system on $S$. Consider the product fusion system $\cal F\x\cal F_S$ on $S\x S$, then the basis element $\beta_{\Delta(S)}$ of Definition \ref{dfnBetas} corresponding to the diagonal subgroup $\Delta(S)\leq S\x S$ satisfies
\[\Phi_{\Delta(P,\ph)}(\beta_{\Delta(S)}) = \frac{\abs S}{\abs{\cal F(P,S)}}\]
for all graphs $\Delta(P,\ph)\leq S\x S$ with $\ph\in \cal F(P,S)$, and $\Phi_{D}(\beta_{\Delta(S)})=0$ for all other subgroups $D\leq S\x S$.
Furthermore $\beta_{\Delta(S)}$ is $\cal F$-generated, symmetric and fully $\cal F$-stable.
\end{lem}

\begin{proof}
$\cal F\x \cal F_S$ is a product of saturated fusion systems and is therefore a saturated fusion system on $S\x S$ according to \cite{BrotoLeviOliver}*{Lemma 1.5}.
The element $\beta_{\Delta(S)}$ is $\cal F\x\cal F_S$-stable by construction, hence considered as an element of $A(S,S)_{(p)}$ it is left $\cal F$-stable according to Lemma \ref{lemStableBiset}.

Next we remark that the $(\cal F\x \cal F_S)$-conjugates of a graph $\Delta(P,\ph)$ with $\ph\in \cal F(P,S)$ are all the other graphs $\Delta(P',\psi)$ with $P' \sim_S P$ and $\psi\in \cal F(P',S)$.
Furthermore, the subgroups of the diagonal $\Delta(S)=\Delta(S,id)$ in $S\x S$ are the graphs $\Delta(P,id)$ for $P\leq S$; and consequently the subgroups of $S\x S$ that are $(\cal F\x \cal F_S)$-subconjugate to $\Delta(S)$ are exactly the graphs $\Delta(P,\ph)$ with $\ph\in \cal F(P,S)$.

According to Proposition \ref{propPLocalBasisDecomp} $\beta_{\Delta(S)}$ only contains orbits of the form $[S\x S/D]$ where $D$ is $(\cal F\x \cal F_S)$-subconjugate to $\Delta(S)$, i.e., where $D$ is a graph $\Delta(P,\ph)$ with $\ph\in \cal F(P,S)$. The element $\beta_{\Delta(S)}$ is therefore $\cal F$-generated.
From Remark \ref{rmkBetas} we know the value of the mark homomorphism on $\beta_{\Delta(S)}$:
\begin{equation}\label{eqCharCoordinates}
\Phi_{D}(\beta_{\Delta(S)}) = \frac{\abs{\Hom_{\cal F\x \cal F_S}(D, \Delta(S,id))} \cdot \abs{S\x S}}{\abs{\Delta(S,id)} \cdot \abs{\Hom_{\cal F\x \cal F_S}(D, S\x S)}}
\end{equation}
for all $D\leq S\x S$. This value is zero unless $D$ is a graph $\Delta(P,\ph)$ with $\ph\in \cal F(P,S)$.
We furthermore note that $\Phi_{\Delta(P,\ph)}(\beta_{\Delta(S)})=\Phi_{\Delta(P,id)}(\beta_{\Delta(S)})$ since $\beta_{\Delta(S)}$ is left $\cal F$-stable. Then \eqref{eqCharCoordinates} becomes
\begin{align*}
\Phi_{\Delta(P,id)}(\beta_{\Delta(S)}) &= \frac{\abs{\Hom_{\cal F\x \cal F_S}(\Delta(P,id), \Delta(S,id))} \cdot \abs{S\x S}}{\abs{\Delta(S,id)} \cdot \abs{\Hom_{\cal F\x \cal F_S}(\Delta(P,id), S\x S)}}.
\end{align*}
The morphisms of $\Hom_{\cal F\x \cal F_S}(\Delta(P,id), S\x S)$ are the pairs $(\ph,c_s)$ where $\ph\in \cal F(P,S)$ and $c_s\in \cal F_S(P,S)$, hence
\[\abs{\Hom_{\cal F\x \cal F_S}(\Delta(P,id), S\x S)}=\abs{\cal F_S(P,S)}\cdot \abs{\cal F(P,S)}.\]
The image of $\Delta(P,id)$ under a morphism $(\ph,c_s)\in \Hom_{\cal F\x \cal F_S}(\Delta(P,id), S\x S)$ is
\[(\ph,c_s)(\Delta(P,id)) = \{(\ph(g),c_s(g)) \mid g\in P\} = \Delta(\lc s P, \ph\circ(c_s)^{-1}).\]
This image lies in $\Delta(S,id)$ if and only if $\ph\circ (c_s)^{-1}=id$, i.e., if $\ph=c_s$. The number of morphisms in $\Hom_{\cal F\x \cal F_S}(\Delta(P,id), \Delta(S,id))$ is therefore simply $\abs{\cal F_S(P,S)}$.

Returning to the expression for $\Phi_{\Delta(P,\ph)}(\beta_{\Delta(S)})=\Phi_{\Delta(P,id)}(\beta_{\Delta(S)})$ we then have
\[\Phi_{\Delta(P,id)}(\beta_{\Delta(S)}) = \frac{\abs{\cal F_S(P,S)} \cdot \abs{S\x S}}{\abs{\Delta(S,id)} \cdot (\abs{\cal F_S(P,S)}\cdot \abs{\cal F(P,S)})} = \frac{\abs{S}}{\abs{\cal F(P,S)}}.\]
For the reflection $(\beta_{\Delta(S)})^\op$ we get
\[\Phi_{\Delta(P,\ph)}((\beta_{\Delta(S)})^\op)= \Phi_{\Delta(\ph P, \ph^{-1})}(\beta_{\Delta(S)}) = \frac{\abs S}{\abs{\cal F(\ph P,S)}} = \frac{\abs S}{\abs{\cal F(P,S)}} = \Phi_{\Delta(P,\ph)}(\beta_{\Delta(S)})\]
for all $\ph\in \cal F$, and $\Phi_{D}((\beta_{\Delta(S)})^\op)=0$ if $D$ is not a graph from $\cal F$. We conclude that $(\beta_{\Delta(S)})^\op = \beta_{\Delta(S)}$ because the fixed point maps agree, and thus $\beta_{\Delta(S)}$ is symmetric and also right $\cal F$-stable.
\end{proof}

\begin{lem}\label{lemStableFixedByCharIdem}
Let $\cal F$ be a saturated fusion system on a $p$-group $S$, and let $\beta_{\Delta(S)}$ be the element of Lemma \ref{lemCharIdemFixedPoints} above. Let $T$ be any $p$-group.

An element $X\in A(S,T)_{(p)}$ is right $\cal F$-stable if and only if $X\circ \beta_{\Delta(S)} = X$. Similarly $X\in A(T,S)_{(p)}$ is left $\cal F$-stable if and only if $\beta_{\Delta(S)}\circ X = X$.
\end{lem}

The content of this lemma is essentially the same as \cite{Ragnarsson}*{Corollary 6.4}. The difference being that \cite{Ragnarsson}*{Corollary 6.4} concerns the characteristic idempotent for $\cal F$, by then shown by Ragnarsson to be unique, while the lemma here is needed to show that $\beta_{\Delta(S)}$ is idempotent in the first place.

\begin{proof}
Because $\beta_{\Delta(S)}$ is right $\cal F$-stable, a product $X\circ \beta_{\Delta(S)}$ is always right $\cal F$-stable as well. Similarly, left-stability of $\beta_{\Delta(S)}$ implies that every product $\beta_{\Delta(S)}\circ X$ is left $\cal F$-stable. This handles both ``if'' directions of the corollary.

To calculate the product $X\circ \omega_{\cal F}$ when $X$ is right $\cal F$-stable, we apply Remark \ref{rmkStabilizationCoeff} to the product fusion system $\cal F\x \cal F_S$ and $\beta_{\Delta(S)}=\pi([S,id])$:
\begin{align*}
 X \circ \beta_{\Delta(S)}
 &= \sum_{\substack{[\Delta(P,\ph)]_{S\x S}\\\text{with } \ph\in \cal F(P,S)}} c_{\Delta(P,\ph)}(\beta_{\Delta(S)}) \cdot (X \circ [P,\ph])
\\ &= \sum_{[P]_S} \left(\sum_{[\Delta(P,\ph)]_{S\x S} \subseteq [\Delta(P,id)]_{\cal F\x \cal F_S}} c_{\Delta(P,\ph)}(\beta_{\Delta(S)}) \right)\cdot (X\circ [P,id])
\\ &\hspace{8cm} \text{since $X$ is right $\cal F$-stable}
\\ &=\sum_{[P]_S} \left(\sum_{[\Delta(P,\ph)]_{S\x S} \subseteq [\Delta(P,id)]_{\cal F\x \cal F_S}} c_{\Delta(P,\ph)}([S,id]) \right)\cdot (X\circ [P,id])
\\ &\hspace{9.7cm}\text{by Remark \ref{rmkStabilizationCoeff}}
\\ &= c_{\Delta(S,id)}([S,id]) \cdot (X\circ [S,id])
\\ &= X.
\end{align*}
By Lemma \ref{lemCharIdemFixedPoints}, we have $(\beta_{\Delta(S)})^\op=\beta_{\Delta(S)}$. Hence, if $X\in A(T,S)_{(p)}$ is left $\cal F$-stable, then $X^\op\in A(S,T)_{(p)}$ is right $\cal F$-stable and
\[\beta_{\Delta(S)}\circ X = (X^\op\circ (\beta_{\Delta(S)})^\op)^\op = (X^\op)^\op = X\]
follows by the right $\cal F$-stable case above.
\end{proof}

\begin{prop}\label{propCharIdemExists}Let $\cal F$ be a saturated fusion system on a finte $p$-group $S$, and let the element $\beta_{\Delta(S)}\in A(S,S)_{(p)}$ be defined with respect to the fusion system $\cal F\x \cal F_S$ on $S\x S$ as in the previous lemmas. We have $\e(\beta_{\Delta(S)})=1$, and $\beta_{\Delta(S)}$ is a characteristic idempotent for $\cal F$.
\end{prop}

\begin{proof}
According to Lemma \ref{lemCharIdemFixedPoints} $\beta_{\Delta(S)}$ is $\cal F$-generated and fully $\cal F$-stable, and by Lemma \ref{lemStableFixedByCharIdem} applied to $X=\beta_{\Delta(S)}$, we have $(\beta_{\Delta(S)})^2=\beta_{\Delta(S)}$. It only remains to check that $\e(\beta_{\Delta(S)})=1$, which is invertible in $\Z_{(p)}$ as required for $\cal F$-characteristic elements.
\begin{align*}
\e(\beta_{\Delta(S)}) &= \sum_{\substack{[\Delta(P,\ph)]_{S\x S}\\\text{with } \ph\in \cal F(P,S)}} c_{\Delta(P,\ph)}(\beta_{\Delta(S)}) \cdot \e([P,\ph])
\\ &= \sum_{\substack{[\Delta(P,\ph)]_{S\x S} \\\text{with } \ph\in \cal F(P,S)}} c_{\Delta(P,\ph)}(\beta_{\Delta(S)}) \cdot \frac{\abs S}{\abs P}
\\&= \sum_{[P]_S} \left(\sum_{[\Delta(P,\ph)]_{S\x S} \subseteq [\Delta(P,id)]_{\cal F\x \cal F_S}} c_{\Delta(P,\ph)}(\beta_{\Delta(S)}) \right)\cdot \frac{\abs S}{\abs P}
\\&= \sum_{[P]_S} \left(\sum_{[\Delta(P,\ph)]_{S\x S} \subseteq [\Delta(P,id)]_{\cal F\x \cal F_S}} c_{\Delta(P,\ph)}([S,id]) \right)\cdot \frac{\abs S}{\abs P} \quad \text{by Remark \ref{rmkStabilizationCoeff}}
\\ &= c_{\Delta(S,id)}([S,id]) \cdot \tfrac{\abs S}{\abs S} = 1.\qedhere
\end{align*}
\end{proof}

\begin{rmk}
For the remainder of the paper, we let $\omega_{\cal F}$ denote the particular $\cal F$-characteristic idempotent $\beta_{\Delta(S)}$. Corollary \ref{corCharIdemUnique} later on proves the uniqueness of characteristic idempotents -- an alternative to Ragnarsson's original uniqueness proof \cite{Ragnarsson}*{Proposition 5.6}.

The original proof \cite{Ragnarsson}*{Proposition 4.9} for the existence of characteristic idempotents by Ragnarsson uses a Cauchy-sequence argument in the $p$-completion $A(S,S)^\wedge_p$ to construct some characteristic idempotent $\omega$ inside $A(S,S)^\wedge_p$. Later arguments then show that $\omega$ is unique and lives already in the $p$-localization $A(S,S)_{(p)}$. Proposition \ref{propCharIdemExists} above instead gives a new explicit construction of $\omega_{\cal F}$ directly inside $A(S,S)_{(p)}$ to begin with.
\end{rmk}

Knowing that $\omega_{\cal F}=\beta_{\Delta(S)}$ is the characteristic idempotent for $\cal F$, it is now straightforward to apply the formulas of Lemma \ref{lemCharIdemFixedPoints} and Proposition \ref{propPLocalBasisDecomp} in order to prove Theorem \ref{thmCharIdemStructure} below.

\begin{mainthm}\label{thmCharIdemStructure}
Let $\cal F$ be a saturated fusion system on a finite $p$-group $S$.
The characteristic idempotent $\omega_{\cal F}\in A(S,S)_{(p)}$ associated to $\cal F$ satisfies:

For all graphs $\Delta(P,\ph)\leq S\x S$ with $\ph\in \cal F(P,S)$, we have
\[\Phi_{\Delta(P,\ph)}(\omega_{\cal F}) = \frac{\abs S}{\abs{\cal F(P,S)}};\]
and $\Phi_{D}(\omega_{\cal F})=0$ for all other subgroups $D\leq S\x S$.
Consequently, if we write $\omega_{\cal F}$ in terms of the transitive bisets in $A(S,S)_{(p)}$, we get the expression
\[\omega_{\cal F} = \hspace{-.3cm}\sum_{\substack{[\Delta(P,\ph)]_{S\x S}\\ \text{with } \ph\in \cal F(P,S)}} \frac {\abs S}{\Phi_{\Delta(P,\ph)}([P,\ph]_S^S)} \Big(\sum_{P\leq Q\leq S}\hspace{-.2cm} \frac{\abs{\{\psi\in \cal F(Q,S) \mid \psi|_P=\ph\}}}{\abs{\cal F(Q,S)}}\cdot \mu(P,Q) \Big)[P,\ph]_S^S,\]
where the outer sum is taken over $(S\x S)$-conjugacy classes of subgroups, and where $\mu$ is the M\"obius function for the poset of subgroups in $S$.
\end{mainthm}

\begin{proof} Since the characteristic idempotent for $\cal F$ is $\omega_{\cal F}=\beta_{\Delta(S)}$, Lemma \ref{lemCharIdemFixedPoints} already states the value of $\Phi_D(\omega_{\cal F})$ for all subgroups $D\leq S\x S$ proving the first part of the theorem.

To determine the decomposition of $\omega_{\cal F}$ as a linear combination of transitive bisets, we apply Proposition \ref{propPLocalBasisDecomp} to the element $\omega_{\cal F}=\beta_{\Delta(S)}$ considered as an $(S\x S)$-set. For $P\leq S$ and $\ph\in \cal F(P,S)$, Proposition \ref{propPLocalBasisDecomp} states that the coefficient of $\omega_{\cal F}$ with respect to the transitive set $[P,\ph]_S^S$ is
\begin{align*}
c_{\Delta(P,\ph)}(\omega_{\cal F}) ={}& \frac 1{\Phi_{\Delta(P,\ph)}([P,\ph])} \left(\sum_{D\geq \Delta(P,\ph)} \Phi_{D}(\omega_{\cal F})\cdot \mu(\Delta(P,\ph),D) \right)
\\ ={}&  \frac 1{\Phi_{\Delta(P,\ph)}([P,\ph])} \left(\sum_{\Delta(Q,\psi)\geq \Delta(P,\ph)} \frac{\abs S}{\abs{\cal F(Q,S)}}\cdot \mu(\Delta(P,\ph),\Delta(Q,\psi)) \right)
\\ ={}& \frac {\abs S}{\Phi_{\Delta(P,\ph)}([P,\ph])} \left(\sum_{Q\geq P} \frac{\abs{\{\psi\in \cal F(Q,S) \mid \psi|_P=\ph\}}}{\abs{\cal F(Q,S)}}\cdot \mu(P,Q) \right).
\end{align*}
In the last step, the equality $\mu(\Delta(P,\ph),\Delta(Q,\psi)) = \mu(P,Q)$ holds because the Möbius function for a poset only depends on the interval between the particular elements, and the poset intervals $(\Delta(P,\ph),\Delta(Q,\psi))$ and $(P,Q)$ are isomorphic.
Since $\omega_{\cal F}$ is $\cal F$-generated, the transitive bisets $[P,\ph]$ with $\ph\in \cal F$ are the only ones that show up in $\omega_{\cal F}$. The coefficients calculated above thus express the characteristic idempotent $\omega_{\cal F}$ as the linear combination in the theorem.
\end{proof}

\subsection{The action of the characteristic idempotent}\label{secCharIdemAction}
In this section we explore how the characteristic idempotent $\omega_{\cal F}$ acts by multiplication on elements of the double Burnside ring and other Burnside modules. Theorem \ref{thmCharIdemMultiplication} gives a precise description of the action of $\omega_{\cal F}$ in terms of the fixed point maps, and in this way we recover the stabilization homomorphism of Theorem \ref{thmStabilizationHom}: The Burnside ring $A(S)_{(p)}$ is isomorphic to the Burnside module $A(1,S)_{(p)}$, and through this isomorphism the stabilization homomorphism of Theorem \ref{thmStabilizationHom} is given by multiplication with $\omega_{\cal F}$ from the left.

We warm up with a result about basis elements for Burnside modules $A(S_1,S_2)_{(p)}$, where $S_1$ and $S_2$ are $p$-groups. We already know that a transitive $(S_1,S_2)$-set $(S_2\x S_1)/D$ only depends on $D$ up to $(S_2\x S_1)$-conjugation, and now we show that when we multiply $(S_2\x S_1)/D$ by characteristic idempotents the result only depends on the subgroup $D$ up to conjugation in the corresponding saturated fusion systems.

\begin{lem}\label{propStableBasis}
Let $\cal F_1$ and $\cal F_2$ be saturated fusion systems on the $p$-groups $S_1$ and $S_2$ respectively, and let $\omega_1\in A(S_1,S_1)_{(p)}$ and $\omega_2\in A(S_2,S_2)_{(p)}$ be their respective characteristic idempotents.

Then for all subgroups $D,C\leq S_2\x S_1$, if $D$ and $C$ are conjugate in $\cal F_2\x\cal F_1$, we have
\[\omega_2\circ [(S_2\x S_1)/D] \circ \omega_1 = \omega_2\circ [(S_2\x S_1)/C]\circ \omega_1\]
in $A(S_1,S_2)_{(p)}$.
\end{lem}

The converse of Lemma \ref{propStableBasis} is also true: If the elements
$\omega_2\circ [(S_2\x S_1)/D] \circ \omega_1$ and $\omega_2\circ [(S_2\x S_1)/C]\circ \omega_1$ are equal, then $D$ and $C$ are conjugate in $\cal F_2\x\cal F_1$. This is an immediate consequence of Corollary \ref{corCharDoubleStabilization} below.

\begin{proof}[Proof of Lemma \ref{propStableBasis}]
Suppose that the subgroups $D,C\leq S_2\x S_1$ are conjugate in $\cal F_2\x \cal F_1$, and let $\ph\in \Hom_{\cal F_2\x \cal F_1}(D,C)$ be an isomorphism.

By definition of $\cal F_2\x \cal F_1$, the homomorphism $\ph$ extends to $(\ph_2\x \ph_1)\colon D_2\x D_1\to C_2\x C_1$ where $D_i$ is the projection of $D$ onto $S_i$, similarly for $C_i$, and where $\ph_i\in \cal F_i(D_i,C_i)$. By assumption, $\ph$ is invertible in $\cal F_2\x \cal F_1$, hence the inverse $\ph^{-1}$ also extends to a homomorphism $C_2\x C_1\to D_2\x D_1$, which shows that $\ph_1$ and $\ph_2$ are invertible in $\cal F_1$ and $\cal F_2$ respectively.
With this we have
\begin{align*}
[(S_2\x S_1)/D] &= [D_2,id]_{D_2}^{S_2}\circ [(D_2\x D_1)/D] \circ [D_1,id]_{S_1}^{D_1}
\\ &= [D_2,id]_{D_2}^{S_2}\circ [C_2, \ph_2^{-1}]_{C_2}^{D_2} \circ [(C_2\x C_1)/C] \circ [D_1,\ph_1]_{D_1}^{C_1} \circ [D_1,id]_{S_1}^{D_1}
\\ &= [C_2, \ph_2^{-1}]_{C_2}^{S_2} \circ [(C_2\x C_1)/C] \circ [D_1,\ph_1]_{S_1}^{C_1}.
\end{align*}
Since $\omega_2$ is right $\cal F_2$-stable, and $\omega_1$ is left $\cal F_1$-stable, it follows that
\begin{align*}
\omega_2\circ [(S_2\x S_1)/D]\circ \omega_1 &= \omega_2\circ [C_2, \ph_2^{-1}]_{C_2}^{S_2} \circ [(C_2\x C_1)/C] \circ [D_1,\ph_1]_{S_1}^{C_1}\circ \omega_1
\\ &= \omega_2\circ [C_2, id]_{C_2}^{S_2} \circ [(C_2\x C_1)/C] \circ [C_1,id]_{S_1}^{C_1}\circ \omega_1
\\ &= \omega_2\circ [(S_2\x S_1)/C] \circ \omega_1.\qedhere
\end{align*}
\end{proof}

\begin{mainthm}\label{thmCharIdemMultiplication}
Let $\cal F_1$ and $\cal F_2$ be saturated fusion systems on finite $p$-groups $S_1$ and $S_2$ respectively, and let $\omega_1\in A(S_1,S_1)_{(p)}$ and $\omega_2\in A(S_2,S_2)_{(p)}$ be the characteristic idempotents.

For every element of the Burnside module $X\in A(S_1,S_2)_{(p)}$, the product $\omega_2\circ X\circ \omega_1$ is right $\cal F_1$-stable and left $\cal F_2$-stable, and satisfies
\[\Phi_{D}(\omega_2\circ X \circ \omega_1) = \frac{1}{\abs{[D]_{\cal F_2\x \cal F_1}}} \sum_{D'\in [D]_{\cal F_2\x \cal F_1}} \Phi_{D'}(X),\]
 for all subgroups $D\leq S_2\x S_1$, where $[D]_{\cal F_2\x \cal F_1}$ is the isomorphism class of $D$ in the product fusion system $\cal F_2\x \cal F_1$ on $S_2\x S_1$.
\end{mainthm}

\begin{cor}\label{corCharDoubleStabilization}
Via the correspondence $A(S_1,S_2)_{(p)}\cong A(S_2\x S_1)_{(p)}$ the map given by $X\mapsto \omega_{\cal F_2}\circ X\circ \omega_{\cal F_1}$ in $A(S_1,S_2)_{(p)}$ coincides with the stabilization map $\pi\colon A(S_2\x S_1)_{(p)} \to A(\cal F_2\x \cal F_1)_{(p)}$ of Theorem \ref{thmStabilizationHom}.

In particular, for each subgroup $D\leq S_2\x S_1$ we have
\[\omega_{\cal F_2}\circ [S_2\x S_1/D]\circ \omega_{\cal F_1} = \beta_D\]
as elements of $A(\cal F_2\x \cal F_1)_{(p)}$.
\end{cor}

Since the identity element $[S_i,id]_{S_i}^{S_i}$ is the characteristic idempotent for the trivial fusion system $\cal F_{S_i}(S_i)$ on $S_i$, Theorem \ref{thmCharIdemMultiplication} and Corollary \ref{corCharDoubleStabilization} also describe all one-sided products $X \circ \omega_1$ and $\omega_2\circ X$.

\begin{proof}[Proof of theorem and corollary]
%Any product $\omega_2\circ X \circ \omega_1$ is right $\cal F_1$-stable by definition since $\omega_1$ is right $\cal F_1$-stable, and similarly we see that $\omega_2\circ X \circ \omega_1$ is left $\cal F_2$-stable.

Consider the element $X\in A(S_1,S_2)_{(p)}$ as an element of $A(S_2\x S_1)_{(p)}$.
The fusion system $\cal F_2\x \cal F_1$ on $S_2\x S_1$ is saturated by \cite{BrotoLeviOliver}*{Lemma 1.5}, and we apply the transfer map $\pi$ of Theorem \ref{thmStabilizationHom}, to get an $(\cal F_2\x \cal F_1)$-stable element $\pi (X)$ satisfying
\[\Phi_{D}(\pi (X)) := \frac{1}{\abs{[D]_{\cal F_2\x \cal F_1}}} \sum_{D'\in [D]_{\cal F_2\x \cal F_1}} \Phi_{D'}(X)\]
for all $D\leq S_2\x S_1$. By Lemma \ref{lemStableBiset}, $\pi(X)$ is left $\cal F_2$-stable and right $\cal F_1$-stable when considered as an element $X\in A(S_1,S_2)_{(p)}$, and we wish to prove that $\omega_2\circ X \circ \omega_1 = \pi(X)$.

By Remark \ref{rmkStabilizationCoeff}, $\pi(X)$ also satisfies
\[\sum_{[D']_{S_2\x S_1} \subseteq [D]_{\cal F_2\x \cal F_1}} c_{D'}(\pi(X)) = \sum_{[D']_{S_2\x S_1} \subseteq [D]_{\cal F_2\x \cal F_1}} c_{D'}(X),\]
or equivalently
\[\sum_{[D']_{S_2\x S_1} \subseteq [D]_{\cal F_2\x \cal F_1}} c_{D'}(\pi(X) - X) = 0.\]
Using Lemma \ref{propStableBasis} we then have
\begin{align*}
&\omega_{2} \circ (\pi(X) - X) \circ \omega_{1}
\\={}& \sum_{[D]_{\cal F_2\x \cal F_1}} \left( \sum_{[D']_{S_2\x S_1} \subseteq [D]_{\cal F_2\x \cal F_1}} c_{D'}(\pi(X) - X) \cdot (\omega_{2}\circ (S_2\x S_1/D') \circ \omega_{1})\right)
\\ ={}& \sum_{[D]_{\cal F_2\x \cal F_1}} \left(\sum_{[D']_{S_2\x S_1} \subseteq [D]_{\cal F_2\x \cal F_1}} c_{D'}(\pi(X) - X) \right) \cdot (\omega_{2}\circ (S_2\x S_1/D) \circ \omega_{1})
\\ ={}& \sum_{[D]_{\cal F_2\x \cal F_1}} 0 \cdot (\omega_{2}\circ (S_2\x S_1/D) \circ \omega_{1}) = 0.
\end{align*}
From which we conclude
\begin{align*}
\omega_2\circ X\circ \omega_1 &= \omega_2\circ \pi(X) \circ \omega_1 = \pi(X),
\end{align*}
where the last equality holds by Lemma \ref{lemStableFixedByCharIdem} since $\pi(X)$ is left $\cal F_2$-stable and right $\cal F_1$-stable.
\end{proof}

\begin{cor}\label{corCharOnesidedStabilization}
Let $\cal F$ be a saturated fusion system on a $p$-group $S$. A set with a left action of $S$ is the same as a $(1,S)$-biset, so the Burnside module $A(1,S)_{(p)}$ is isomorphic to the Burnside ring $A(S)_{(p)}$.
Through this isomorphism left multiplication with $\omega_{\cal F}$ in $A(1,S)_{(p)}$ coincides with the stabilization homomorphism $\pi\colon A(S)_{(p)}\to A(\cal F)_{(p)}$ of Theorem \ref{thmStabilizationHom}.
\end{cor}

\begin{proof}
The subgroups of $S\x 1$ are all on the form $Q\x 1$ for som $Q\leq S$, and the characteristic idempotent for the unique fusion system on the trivial group is just $[1,id]_1^1=[pt]_1^1$. By Theorem \ref{thmCharIdemMultiplication} we then have
\[\Phi_{Q\x 1}(\omega_{\cal F} \circ X) = \frac{1}{\abs{[Q]_{\cal F}}} \sum_{Q'\in [Q]_{\cal F}} \Phi_{Q'\x 1}(X) = \Phi_Q(\pi (X))\]
for all $X\in A(1,S)_{(p)}$,
so $\omega_{\cal F}\circ X = \pi(X)$ as claimed.
\end{proof}

\begin{dfn}\label{dfnDoubleBurnside}
For saturated fusion systems $\cal F_1,\cal F_2$ on $p$-groups $S_1,S_2$, we define the \emph{Burnside module} $A(\cal F_1,\cal F_2)_{(p)}$ as the $\Z_{(p)}$-submodule of $A(S_1,S_2)_{(p)}$ consisting of the elements that are right $\cal F_1$-stable and left $\cal F_2$-stable.
\end{dfn}

\begin{rmk}
The elements $\omega_{\cal F_2}\circ [(S_2\x S_1)/D] \circ \omega_{\cal F_1}$ generate $A(\cal F_1,\cal F_2)_{(p)}$ over $\Z_{(p)}$. By Corollary \ref{corCharDoubleStabilization} the element $\omega_{\cal F_2}\circ [(S_2\x S_1)/D] \circ \omega_{\cal F_1}$ actually corresponds to the element $\beta_D\in A(\cal F_2\x \cal F_1)_{(p)}$, so it follows that the elements $\{\omega_{\cal F_2}\circ [(S_2\x S_1)/D] \circ \omega_{\cal F_1}\mid D\leq S_2\x S_1\}$ form a $\Z_{(p)}$-basis for the Burnside module $A(\cal F_1,\cal F_2)_{(p)}$. Two subgroups $C$ and $D$ give the same basis element if and only if $C$ and $D$ are conjugate in $\cal F_2\x \cal F_1$. The existence of such basis elements nicely generalizes the basis we have for the Burnside modules of groups.

By restricting the composition map for groups, the Burnside modules $A(\cal F_1,\cal F_2)_{(p)}$ come equipped with an associative composition
\[\circ \colon A(\cal F_2,\cal F_3)_{(p)}\x A(\cal F_1,\cal F_2)_{(p)} \to A(\cal F_1,\cal F_3)_{(p)},\]
For each saturated fusion system $\cal F$, the characteristic idempotent $\omega_{\cal F}\in A(\cal F,\cal F)_{(p)}$ is an identity element for the composition, in particular the double Burnside ring $A(\cal F,\cal F)_{(p)}$ has $\omega_{\cal F}$ as its $1$-element.
\end{rmk}

Multiplication with characteristic idempotents $\omega_1$ and $\omega_2$ defines a map $A(S_1,S_2)_{(p)}\to A(\cal F_1,\cal F_2)_{(p)}$, and as in Theorem \ref{thmStabilizationHom}, it is a homomorphism of modules:

\begin{prop}\label{propDoubleStabilizationHom}
Let $\cal F_1$ and $\cal F_2$ be saturated fusion systems on $p$-groups $S_1$ and $S_2$ respectively, and let $\omega_1\in A(S_1,S_1)_{(p)}$ and $\omega_2\in A(S_2,S_2)_{(p)}$ be their characteristic idempotents.

Then the map $\pi\colon A(S_1,S_2)_{(p)}\to A(\cal F_1,\cal F_2)_{(p)}$ given by $\pi(X):= \omega_2\circ X\circ \omega_1$ is a homomorphism of left $A(\cal F_2,\cal F_2)_{(p)}$-modules and right $A(\cal F_1,\cal F_1)_{(p)}$-modules.
\end{prop}

\begin{proof}
We only show that $\pi$ is a homomorphism of right $A(\cal F_1,\cal F_1)_{(p)}$-modules, since the other case follows by symmetry.

Let $X\in A(S_1,S_2)_{(p)}$ be given, and let $Z\in A(\cal F_1,\cal F_1)_{(p)}$ be a fully $\cal F_1$-stable element of $A(S_1,S_1)_{(p)}$. Then $\cal F_1$-stability ensures that $\omega_1\circ Z = Z \circ \omega_1=Z$ by Lemma \ref{lemStableFixedByCharIdem}; hence we get
\[\pi(X\circ Z) = \omega_2\circ X\circ Z\circ \omega_1 =\omega_2\circ X\circ \omega_1 \circ Z = \pi(X)\circ Z.\qedhere\]
\end{proof}

\section{The Burnside ring embeds in the double Burnside ring}\label{secRingOfChar}
In this section we show that the ``one-sided'' Burnside ring $A(\cal F)_{(p)}$ of sections \ref{secBurnsideFusion} and \ref{secStabilization} always embeds in the double Burnside ring $A(\cal F,\cal F)_{(p)}$ of Definition \ref{dfnDoubleBurnside} above.
In fact, Theorem \ref{thmSingleEmbedsInDouble} states that $A(\cal F)_{(p)}$ is isomorphic to the subring generated by all $\cal F$-characteristic elements. Through this isomorphism we can describe the structure of all the $\cal F$-characteristic elements.

The isomorphism between the ``one-sided'' Burnside ring and a subring of the double Burnside ring is inspired by a similar result for finite groups, where the Burnside ring $A(G)$ embeds in the double Burnside ring $A(G,G)$.
Let us therefore first analyze the situation for Burnside rings of $p$-groups and see what might be generalized to fusion systems:
\begin{ex}\label{exGroupSingleEmbedsInGroupDouble}
Let $S$ be a finite $p$-group, and define $\iota\colon A(S)\to A(S,S)$ on $S$-sets by
\[\iota(X) := X\x S \text{ equipped with the action }a(x,s)b = (ax,axb).\]
It is easy to check that $\iota(X\x Y)=X\x Y\x S$ is isomorphic as bisets to $\iota(X)\circ \iota(Y) =\linebreak (X\x S)\x_S (Y\x S)$, where the element $(x,y,s)\in \iota(X\x Y)$ corresponds to the class of $(x,1,y,s)$ in $\iota(X)\circ \iota(Y)$. Hence we get an injective ring homomorphism $\iota\colon A(S)\to A(S,S)_{(p)}$.

On transitive sets we have $\iota([S/P])= [P,id]$, where the class of $(s,t)$ in $(S/P)\x S$ corresponds to the class of $(s,s^{-1}t)$ in $S\x_P S$. Therefore $\iota$ embeds $A(S)$ as the unital subring of $A(S,S)$ generated by $[P,id]$ for $P\leq S$.
The basis elements $[P,id]$ are precisely those basis elements $[(S\x S)/D]$ where $D=\Delta(P,c_s)$ is the graph of an $S$-conjugation map -- recall that the subgroup $\Delta(P,c_s)$ is only determined up to $(S\x S)$-conjugation, so $[P,c_s]=[P,id]$.
The subring generated by $[P,id]$ for $P\leq S$, is therefore the ring $A_{\cal F_S}(S,S)$ of all $\cal F_S$-generated elements (see definition \ref{dfnFgenerated}), where $\cal F_S$ is the trivial fusion system on $S$. This suggests that we should consider the $\cal F$-generated elements for general fusion systems.

Finally the inverse of $\iota$ is the map $q\colon A_{\cal F_S}(S,S)\to A(S)$ given on bisets by $X \mapsto X/S$. Here we eliminate the right $S$-action by quotienting out, equivalently this can be expressed by the multiplication $X\mapsto X \circ [(S\x 1)/(S\x 1)]$ from $A_{\cal F_S}(S,S)$ to $A(1,S)$. It is trivial to check that $q(\iota(X))=X$ for all $S$-sets $X\in A(S)$.
The map $X\mapsto X/S$ does not preserve the multiplication on all of $A(S,S)$, so it is not clear that $q$ defined on the subring $A_{\cal F_S}(S,S)$ respects multiplication, but this must be true since $q=\iota^{-1}$. A similar situation occurs in Theorem \ref{thmSingleEmbedsInDouble}: We state the theorem for the simple map $q$ where we quotient out the right $S$-action, but to actually see that $q$ is a ring homomorphism we construct the inverse $\iota$ as a ring homomorphism from the start.
\end{ex}

For a saturated fusion system $\cal F$ on $S$, the double Burnside ring $A(\cal F,\cal F)_{(p)}$ was defined to be the subring of $A(S,S)_{(p)}$ consisting of the elements that are both left and right $\cal F$-stable.
Example \ref{exGroupSingleEmbedsInGroupDouble} suggests that we should look at those elements of $A(\cal F,\cal F)_{(p)}$ that are in addition $\cal F$-generated:

\begin{dfn}
We define
\[A^{\text{char}}(\cal F)_{(p)}:=A(\cal F,\cal F)_{(p)}\cap A_{\cal F}(S,S)_{(p)}\]
as the subring formed by all elements that are $\cal F$-stable as well as $\cal F$-generated.

Hence we have a sequence of inclusions of subrings
\[A^{\text{char}}(\cal F)_{(p)} \subseteq A(\cal F,\cal F)_{(p)}\subseteq A(S,S)_{(p)}.\]
The last inclusion is not unital since $\omega_{\cal F}$ is the multiplicative identity of the first two rings, and $[S,id]_S^S$ is the identity of $A(S,S)_{(p)}$.
\end{dfn}

We use the notation $A^{\text{char}}(\cal F)_{(p)}$ for this particular subring because the following proposition shows that $A^{\text{char}}(\cal F)_{(p)}$ is generated, over $\Z_{(p)}$, by all the $\cal F$-characteristic elements in $A(S,S)_{(p)}$. Note that not all elements of $A^{\text{char}}(\cal F)_{(p)}$ are $\cal F$-characteristic, but the non-characteristic elements of $A^{\text{char}}(\cal F)_{(p)}$ form a proper $\Z_{(p)}$-submodule.

\begin{prop}\label{propCharElemBasis}
Let $\cal F$ be a saturated fusion systems on a $p$-group $S$, and let $A^{\text{char}}(\cal F)_{(p)}$ be defined as above.
Then $A^{\text{char}}(\cal F)_{(p)}$ is also the subring of $A(S,S)_{(p)}$ generated by the $\cal F$-characteristic elements, and it has a $\Z_{(p)}$-basis consisting of the elements $\beta_{\Delta(P,id)}=\omega_{\cal F}\circ [P,id] \circ \omega_{\cal F}$, which are in one-to-one correspondence with the $\cal F$-conjugacy classes of subgroups $P\leq S$.

The characteristic elements of $\cal F$ are those elements $X\in A^{\text{char}}(\cal F)_{(p)}$ where the coefficient of $X$ at the basis element $\beta_{\Delta(S,id)}=\omega_{\cal F}$ is invertible in $\Z_{(p)}$.
\end{prop}

\begin{proof}
We first claim that $A^{\text{char}}(\cal F)_{(p)} = \omega_{\cal F} \circ A_{\cal F}(S,S)_{(p)} \circ \omega_{\cal F}$, where $A_{\cal F}(S,S)_{(p)}$ is the subring of $\cal F$-generated elements in $A(S,S)_{(p)}$.
Each element in $\omega_{\cal F} \circ A_{\cal F}(S,S)_{(p)} \circ \omega_{\cal F}$ is $\cal F$-stable and a product of $\cal F$-generated elements (hence $\cal F$-generated as well), so it is contained in $A^{\text{char}}(\cal F)_{(p)}$.

Conversely, suppose $X\in A^{\text{char}}(\cal F)_{(p)}$.
Because $X$ is $\cal F$-stable, we have
$X= \omega_{\cal F} \circ X \circ \omega_{\cal F}$
by Lemma \ref{lemStableFixedByCharIdem}, so $X$ lies in the product $\omega_{\cal F} \circ A_{\cal F}(S,S)_{(p)} \circ \omega_{\cal F}$.
We conclude that we have $A^{\text{char}}(\cal F)_{(p)} = \omega_{\cal F} \circ A_{\cal F}(S,S)_{(p)} \circ \omega_{\cal F}$ as claimed.

We know that $A_{\cal F}(S,S)_{(p)}$ is generated by the sets $[P,\ph]$ with $\ph\in\cal F(P,S)$ by definition. Hence $A^{\text{char}}(\cal F)$ is generated by the elements $\omega_{\cal F}\circ [P,\ph]\circ \omega_{\cal F}$ with $\ph\in \cal F(P,S)$, and by Lemma \ref{propStableBasis} we have $\omega_{\cal F}\circ [P,\ph]\circ \omega_{\cal F}=\omega_{\cal F}\circ [P,id]\circ \omega_{\cal F} = \beta_{\Delta(P,id)}$ as elements of $A(\cal F\x \cal F)_{(p)}$. So the elements $\beta_{\Delta(P,id)}$ generate $A^{\text{char}}(\cal F)_{(p)}$ and are linearly independent over $\Z_{(p)}$ since they are already part of a basis for the double Burnside ring $A(\cal F,\cal F)_{(p)}$. Two basis elements $\beta_{\Delta(P,id)}$ and $\beta_{\Delta(Q,id)}$ are equal exactly when $\Delta(P,id)$ and $\Delta(Q,id)$ are $(\cal F\x \cal F)$-conjugate, which happens if and only if $P$ and $Q$ are $\cal F$-conjugate.

The elements $X\in A^{\text{char}}(\cal F)_{(p)}$ are already $\cal F$-stable and $\cal F$-generated, so the only extra condition that $\cal F$-characteristic elements must satisfy is that $\e(X)$ is invertible in $\Z_{(p)}$, i.e., $\e(X)\not\equiv 0\pmod p$ in $\Z_{(p)}$. Any basis element of $A^{\text{char}}(\cal F)_{(p)}$ other than $\beta_{\Delta(S,id)}$ is of the form $\omega_{\cal F}\circ [P,id] \circ \omega_{\cal F}$ with $P< S$. Because $\e(\omega_{\cal F})=1$, by Proposition \ref{propCharIdemExists}, we therefore have
\[\e(\omega_{\cal F}\circ [P,id] \circ \omega_{\cal F}) =1\cdot \e([P,id])\cdot 1 =\frac{\abs S}{\abs P} \equiv 0 \pmod p\]
for all $P< S$.
So whether $\e(X)\not\equiv 0 \pmod p$ depends only on the coefficient of $X$ at the basis element $\beta_{\Delta(S,id)}=\omega_{\cal F}$.

Finally, it is clear that every element in $A^{\text{char}}(\cal F)_{(p)}$ can be written as the sum of one or more elements with $\e(X)\not\equiv 0 \pmod p$ in $\Z_{(p)}$, so $A^{\text{char}}(\cal F)_{(p)}$ is additively generated by the characteristic elements.\end{proof}

\begin{lem}\label{lemPhiOfIota} Let $\iota^S\colon A(S)_{(p)} \to A(S,S)_{(p)}$ be the injective ring homomorphism of Example \ref{exGroupSingleEmbedsInGroupDouble} mapping $[S/P]\mapsto [P,id]$.
For every $X\in A(S)_{(p)}$ and subgroup $D\leq S\x S$, we have $\Phi_{D}(\iota^S(X))=0$ unless $D$ is $(S\x S)$-conjugate to $\Delta(Q,id)$ for some $Q\leq S$. In that case
\[\Phi_{\Delta(Q,id)}(\iota^S (X)) = \Phi_{Q}(X)\cdot \abs{C_S(Q)}.\]
Furthermore, $\iota^S(X)$ is symmetric for all $X\in A(S)_{(p)}$, i.e., $\iota^S(X)^\op =\iota^S(X)$.
\end{lem}

\begin{proof}
By linearity in $X\in A(S)_{(p)}$, it is enough to prove the lemma for basis elements $[S/P]\in A(S)_{(p)}$, where $P\leq S$. The symmetry is obvious since $\iota^S([S/P])=[P,id]$, which is symmetric.

Since $\iota^S([S/P])=[P,id]$, we apply the formula \eqref{eqPhiOnBasis} for the fixed-point homomorphisms on basis elements: For $D\leq S\x S$ we have $\Phi_D([P,id])=0$ unless $D$ is $(S\x S)$-subconjugate to $\Delta(P,id)$. The subgroups of $\Delta(P,id)$ are $\Delta(R,id)$ for $R\leq P$, hence $D$ has to be of the form $\Delta(Q,c_s)$ for $Q\leq S$ and $s\in S$, which is $(S\x S)$-conjugate to $\Delta(Q,id)$.
For the graph $\Delta(Q,id)$ we then have
\begin{align*}
\Phi_{\Delta(Q,id)}(\iota^S([S/P])) &=\frac{\abs{N_{S\x S} (\Delta(Q,id), \Delta(P,id))}}{\abs{\Delta(P,id)}}
\\ &= \frac{\abs{\{(s,t) \mid s,t\in N_S(Q,P) \text{ and } c_s = c_t \in \Hom_S(Q,P)\}}}{\abs P}
\\ &=\frac{\abs{N_S(Q,P)}}{\abs{P}}\cdot \abs{C_S(Q)} = \Phi_{Q}([S/P]) \cdot \abs{C_S(Q)}.\qedhere
\end{align*}
\end{proof}

\begin{lem}\label{lemCharIdemTimesBeta}
Let $\cal F$ be a saturated fusion system on a $p$-group $S$. For all basis elements $\beta_{P}\in A(\cal F)_{(p)}$ it holds that \[\omega_{\cal F}\circ \iota^S( \beta_{P}) \circ \omega_{\cal F} = \omega_{\cal F}\circ \iota^S( \beta_{P}) = \iota^S( \beta_{P}) \circ \omega_{\cal F} = \beta_{\Delta(P,id)}.\]

By linearity, we get for all $X\in A(\cal F)_{(p)}$ that $\omega_{\cal F}\circ \iota^S( X) \circ \omega_{\cal F} = \omega_{\cal F}\circ \iota^S( X) = \iota^S( X) \circ \omega_{\cal F}$.
\end{lem}

\begin{proof}
Because the basis element $\beta_{\Delta(P,id)}\in A^{\text{char}}(\cal F)_{(p)}$ is $\cal F$-generated, we have that $\Phi_D(\beta_{\Delta(P,id)})=0$ unless $D$ has the form $\Delta(Q,\psi)$ with $\psi\in\cal F(Q,S)$, and because $\beta_{\Delta(P,id)}$ is $\cal F$-stable, we have $\Phi_{\Delta(Q,\psi)}(\beta_{\Delta(P,id)})=\Phi_{\Delta(Q,id)}(\beta_{\Delta(P,id)})$ when $\psi\in \cal F(Q,S)$. Considered as an element of $A(\cal F\x \cal F)_{(p)}$ we know these fixed point values from Proposition \ref{thmPLocalBurnsideBasis}:\enlargethispage{.5cm}
\begin{align*}
\Phi_{\Delta(Q,id)}(\beta_{\Delta(P,id)}) &= \frac{\abs{\Hom_{\cal F\x\cal F}(\Delta(Q,id),\Delta(P,id))}\cdot \abs{S\x S}}{\abs{\Delta(P,id)}\cdot \abs{\Hom_{\cal F\x\cal F}(\Delta(Q,id),S\x S)}}
\\ &= \frac{\abs{\cal F(Q,P)}\cdot \abs{S}^2}{\abs{P}\cdot \abs{\cal F(Q,S)}^2} = \Phi_Q(\beta_P) \cdot \frac{\abs S}{\abs{\cal F(Q,S)}}.
\end{align*}
For the product $\omega_{\cal F} \circ \iota^S(\beta_P)$ we apply Theorem \ref{thmCharIdemMultiplication} to give us
\begin{equation}\label{eqIncludedBeta}
\Phi_{\Delta(Q,\psi)}(\omega_{\cal F} \circ \iota^S(\beta_P)) = \frac1{\abs{[\Delta(Q,\psi)]_{\cal F\x \cal F_S}}} \sum_{\Delta(Q',\psi')\in [\Delta(Q,\psi)]_{\cal F\x \cal F_S}} \Phi_{\Delta(Q',\psi')}(\iota^S(\beta_P)),
\end{equation}
where $\cal F_S$ is the trivial fusion system on $S$.
By Lemma \ref{lemPhiOfIota}, $\Phi_{\Delta(Q',\psi')}(\iota^S(\beta_P))=0$ unless $\Delta(Q',\psi')$ is $(S\x S)$-conjugate to $\Delta(Q',id)$. Since $Q'\sim_S Q$ for all subgroups $\Delta(Q',\psi')\in [\Delta(Q,\psi)]_{\cal F\x \cal F_S}$, we conclude that all summands are zero unless $\Delta(Q,id)\in [\Delta(Q,\psi)]_{\cal F\x \cal F_S}$. Hence $\Delta(Q,\psi)$ should be conjugate to  $\Delta(Q,id)$ inside $\cal F\x \cal F_S$, i.e., $\psi$ must lie in $\cal F$.

In this case we have, by left $\cal F$-stability of $\omega_{\cal F} \circ \iota^S(\beta_P)$, that
\[\Phi_{\Delta(Q,\psi)}(\omega_{\cal F} \circ \iota^S(\beta_P)) = \Phi_{\Delta(Q,id)}(\omega_{\cal F} \circ \iota^S(\beta_P)).\] We still get $\Phi_{\Delta(Q',\psi')}(\iota^S(\beta_P))=0$ unless $\Delta(Q',\psi')$ is actually $(S\x S)$-conjugate to $\Delta(Q',id)$ and $\Delta(Q,id)$. In \eqref{eqIncludedBeta} we can therefore omit all the summands that are zero, and we get
\begin{align*}
\MoveEqLeft \Phi_{\Delta(Q,\psi)}(\omega_{\cal F} \circ \iota^S(\beta_P))
\\ &= \Phi_{\Delta(Q,id)}(\omega_{\cal F} \circ \iota^S(\beta_P))
\\ &= \frac1{\abs{[\Delta(Q,id)]_{\cal F\x \cal F_S}}} \sum_{\Delta(Q',\psi')\in [\Delta(Q,id)]_{S\x S}} \Phi_{\Delta(Q',\psi')}(\iota^S(\beta_P))
\\ &= \frac{\abs{[\Delta(Q,id)]_{S\x S}}}{\abs{[\Delta(Q,id)]_{\cal F\x \cal F_S}}}\cdot \Phi_{\Delta(Q,id)}(\iota^S(\beta_P))
\\ &= \frac{\abs{\Hom_{S\x S}(\Delta(Q,id),S\x S)} \cdot \abs{\Aut_{\cal F\x \cal F_S}(\Delta(Q,id))}}{\abs{\Aut_{S\x S}(\Delta(Q,id))}\cdot \abs{\Hom_{\cal F\x \cal F_S}(\Delta(Q,id),S\x S)}}\cdot \Phi_{Q}(\beta_{P}) \cdot \abs{C_S(Q)}
\\ &= \frac{\frac{\abs S^2}{\abs{C_S(Q)}^2} \cdot \frac{\abs{N_S(Q)}}{\abs{C_S(Q)}}}{\frac{\abs{N_S(Q)}}{\abs{C_S(Q)}} \cdot (\abs{\cal F(Q,S)}\cdot\frac{\abs S}{\abs{C_S(Q)}})} \cdot \Phi_{Q}(\beta_{P}) \cdot \abs{C_S(Q)}
\\ &= \Phi_{Q}(\beta_{P})\cdot \frac{\abs S}{\abs{\cal F(Q,S)}} = \Phi_{\Delta(Q,\psi)}(\beta_{\Delta(P,id)}).
\end{align*}
This shows that $\omega_{\cal F} \circ \iota^S(\beta_P) = \beta_{\Delta(P,id)}$; and by symmetry of $\omega_{\cal F}$ we have
\[\beta_{\Delta(P,id)}=(\beta_{\Delta(P,id)})^\op = (\omega_{\cal F} \circ \iota^S(\beta_P))^\op = \iota^S(\beta_{P})^\op \circ \omega_{\cal F}^\op = \iota^S(\beta_{P}) \circ \omega_{\cal F}.\]
Finally, $\omega_{\cal F}\circ (\iota^S(\beta_{P}) \circ \omega_{\cal F}) = \omega_{\cal F}\circ (\omega_{\cal F} \circ \iota^S(\beta_P)) = \omega_{\cal F} \circ \iota^S(\beta_P)$.
\end{proof}

\begin{mainthm}\label{thmSingleEmbedsInDouble}
Let $\cal F$ be a saturated fusion system on a finite $p$-group $S$.

Then the collapse map $q\colon A^{\text{char}}(\cal F)_{(p)} \to A(\cal F)_{(p)}$, which quotients out the right $S$-action, is an isomorphism of rings, and it sends the basis element $\beta_{\Delta(P,id)}$ of $A^{\text{char}}(\cal F)_{(p)}$ to the basis element $\beta_{P}$ of $A(\cal F)_{(p)}$.
\end{mainthm}

\begin{proof}
For a biset $X$ the quotient $X/S$ is the same as the product $X\x_S \mathord{pt}$, so the collapse map $q\colon A(S,S)_{(p)}\to A(S)_{(p)}$ is alternatively given as right-multiplication with the one-point $(1,S)$-biset $[pt]_1^S$.
The one-point biset has $\Phi_D([pt]_1^S)=1$ for all $D\leq S\x 1$, and by Theorem \ref{thmCharIdemMultiplication} we then also have $\Phi_D(\omega_{\cal F} \circ [pt]_1^S) = 1$ for all $D\leq S\x 1$, so $\omega_{\cal F}\circ [pt]_1^S=[pt]_1^S$.

If we apply the collapse map $q$ to the basis elements $\beta_{\Delta(P,id)}=\omega_{\cal F}\circ [P,id]\circ \omega_{\cal F}$ of $A^{\text{char}}(\cal F)_{(p)}$ we therefore get
\begin{align*}
q(\beta_{\Delta(P,id)}) &= \omega_{\cal F} \circ [P,id]_S^S \circ \omega_{\cal F} \circ [pt]_1^S
\\ &= \omega_{\cal F} \circ [P,id]_S^S \circ [pt]_1^S = \omega_{\cal F} \circ [S/P]_1^S.
\end{align*}
By Corollary \ref{corCharOnesidedStabilization} multiplication with $\omega_{\cal F}$ in $A(1,S)_{(p)}$ is the same as the stabilization map of Theorem \ref{thmStabilizationHom}, so $q(\beta_{\Delta(P,id)})=\omega_{\cal F} \circ [S/P]_1^S=\beta_P$ as elements of $A(S)_{(p)}$.

Now we define a $\Z_{(p)}$-homomorphism $\iota^{\cal F} \colon A(\cal F)_{(p)} \to A^{\text{char}}(\cal F)_{(p)}$ by
\[\iota^{\cal F}(X) = \omega_{\cal F} \circ \iota^S(X) \circ \omega_{\cal F},\]
and by Lemma \ref{lemCharIdemTimesBeta} we then have $\iota^{\cal F}(\beta_P) = \beta_{\Delta(P,id)}$.
Because $q$ sends $\beta_{\Delta(P,id)}\in A^{\text{char}}(\cal F)_{(p)}$ to $\beta_P\in A(\cal F)_{(p)}$, and $\iota^{\cal F}$ sends it back again, the two maps $q$ and $\iota^{\cal F}$ are inverse isomorphisms of $\Z_{(p)}$-modules $A(\cal F)_{(p)}$ and $A^{\text{char}}(\cal F)_{(p)}$.

Finally, we recall that $\iota^S$ is a ring homomorphism, and apply Lemma \ref{lemCharIdemTimesBeta} to show that all elements $X,Y\in A(\cal F)_{(p)}$ satisfy
\[(\omega_{\cal F}\circ \iota^S( X) \circ \omega_{\cal F})\circ (\omega_{\cal F}\circ \iota^S (Y) \circ \omega_{\cal F}) = \omega_{\cal F}\circ \iota^S( X) \circ \iota^S( Y) \circ \omega_{\cal F} = \omega_{\cal F}\circ \iota^S (X Y) \circ \omega_{\cal F}.\]
Hence $\iota^\cal F$ preserves multiplication, and consequently the inverse $q\colon A^{\text{char}}(\cal F)_{(p)}\to A(\cal F)_{(p)}$ does as well.
\end{proof}

Via the ring isomorphism $A^{\text{char}}(\cal F)_{(p)}\cong A(\cal F)_{(p)}$ we can translate any question about the structure of $\cal F$-characteristic elements into a question about the Burnside ring $A(\cal F)_{(p)}$. For instance we get an alternative to Ragnarsson's proof in \cite{Ragnarsson}*{Proposition 5.6} that characteristic idempotents are unique: Suppose until now that we have used $\omega_{\cal F}$ only to denote the particular characteristic idempotent $\beta_{\Delta(S)}$ of Proposition \ref{propCharIdemExists}, then we can determine all other characteristic idempotents in $A^{\text{char}}(\cal F)_{(p)}$ by studying idempotents in $A(\cal F)_{(p)}$ instead:
\begin{cor}\label{corCharIdemUnique}
Let $\cal F$ be a saturated fusion system on a finite $p$-group $S$.
The only idempotents of $A(\cal F)_{(p)}$ are $0$ and the $1$-element $[S/S]$. Hence it follows that $A^{\text{char}}(\cal F)_{(p)}$ has exactly one non-zero idempotent, and therefore the characteristic idempotent $\omega_{\cal F}$ is unique.
\end{cor}

\begin{proof}
By Proposition \ref{propYoshidaFusion} the Burnside ring $A(\cal F)_{(p)}$ fits in a short-exact sequence of $\Z_{(p)}$-modules
\[0\to A(\cal F)_{(p)} \xto{\Phi} \free {\cal F}_{(p)} \xto{\Psi} Obs(\cal F)_{(p)} \to 0.\]
Here $\Phi$ is the mark homomorphism, $Obs(\cal F)_{(p)}$ is the group
\[Obs(\cal F)_{(p)} = \prod_{\substack{[P]_{\cal F}\in\ccset{\cal F}\\ P\text{ f.n.}}} (\Z / \abs{W_S P}\Z),\]
and $\Psi$ is given by the $[P]_{\cal F}$-coordinate functions
\[\Psi_{P}(\xi) := \sum_{\bar s\in W_S P} \xi_{\gen s P} \pmod {\abs{W_S P}},\]
when $P$ is fully $\cal F$-normalized, and $\xi_{\gen s P}$ denotes the $[\gen s P]_{\cal F}$-coordinate of an element $\xi\in \free {\cal F}_{(p)} = \prod_{[P]\in \ccset {\cal F}} \Z_{(p)}$.

Let $e$ be an idempotent in $A(\cal F)_{(p)}$, then since $\Phi$ is a ring homomorphism, the fixed point vector $\Phi(e)$ must be idempotent in the product ring $\free{\cal F}_{(p)}$. Since $\Phi(e)$ is an element of a product ring, it is idempotent if and only if each coordinate $\Phi_Q(e)$ is idempotent in $\Z_{(p)}$. The only idempotents of $\Z_{(p)}$ are $0$ and $1$, so $e\in A(\cal F)_{(p)}$ is idempotent if and only if we have $\Phi_Q(e)\in \{0,1\}$ for all $Q\leq S$.

Let the top coordinate $\Phi_S(e)$ be fixed as either $0$ or $1$, then we will prove by induction on the index of $Q\leq S$ that (under the assumption of idempotence) the coordinate $\Phi_Q(e)$ is determined by $\Phi_S(e)$. Suppose that $Q< S$, and that $\Phi_{R}(e)$ is determined for all $R$ with $\abs R > \abs Q$. Then because $\Psi\Phi=0$, the fixed points must satisfy
\[\sum_{\bar s\in W_S Q} \Phi_{\gen s Q}(e) \equiv 0 \pmod{\abs{W_S Q}},\]
or if we isolate $\Phi_Q(e)$:
\[\Phi_Q(e) \equiv - \sum_{\substack{\bar s\in W_S Q\\ \bar s \neq 1}} \Phi_{\gen s Q}(e) \pmod{\abs{W_S Q}}.\]
We have $\abs{\gen s Q} > \abs Q$ for all $s\in N_S Q$ with $s\not\in Q$, so all the numbers $\Phi_{\gen s Q}(e)$ are already determined by induction. In addition $Q< S$ implies $Q< N_S Q$, so $\abs{W_S Q}\geq 2$, and thus $\Phi_Q(e)=0$ and $\Phi_Q(e)=1$ cannot both satisfy the congruence relation.

We conclude that once $\Phi_S(e)$ is fixed, there is at most one possibility for $e$. The empty set $0=[\Ø]$ is idempotent and satisfies $\Psi_S(0)=0$, and the one point set $[S/S]$ is idempotent and satisfies $\Phi_S([S/S])=1$, so both possibilities exist.
\end{proof}

\section{On the composition product of saturated fusion systems}\label{secCompositionProducts}
In this final section we apply the earlier Theorems \ref{thmCharIdemStructure} and \ref{thmCharIdemMultiplication} about characteristic idempotents to a conjecture of Park-Ragnarsson-Stancu in \cite{ParkRagnarssonStancu} concerning composition products of fusion systems and how to characterize them in terms of characteristic idempotents.
Theorem \ref{thmCompositionProducts} states the precise conditions under which the conjecture of Park-Ragnarsson-Stancu holds. We proceed to give a counterexample to the general conjecture as well as prove a special case. Finally we suggest a revised definition of composition products (see Definition \ref{dfnRevisedCompositionProd}) with respect to which  the conjecture holds in general.

Let $\cal F$ be a fusion system on a $p$-group $S$, and let $\cal H,\cal K$ be fusion subsystems on subgroups $R,T\leq S$ respectively. We say that $\cal F$ is the \emph{weak composition product} of $\cal H$ and $\cal K$, written $\cal F\approx\cal H\cal K$, if $S=RT$ and for all subgroups $P\leq T$ it holds that every morphism $\ph\in \cal F(P,R)$ can be written as a composition $\ph=\psi\rho$ such that $\psi$ is a morphism of $\cal H$ and $\rho$ is a morphism of $\cal K$. Park-Ragnarsson-Stancu use the term ``composition product'' instead of ``weak composition product'', but as we shall see in Theorem \ref{thmCompositionProducts} the weak composition product lacks one additional condition on $\cal F$, $\cal H$ and $\cal K$. We reserve the term ``composition product'' for the revised Definition \ref{dfnRevisedCompositionProd}.

For a finite group $G$ with subgroups $H,K\leq G$, we can ask whether $G=HK$, i.e., if every $g\in G$ can be written as $g=hk$ with $h\in H$ and $k\in K$. It turns out that the answer to this question is detected by the structure of $G$ as an $(K,H)$-biset. With a little thought one can show that $G=HK$ if an only if the $(K,H)$-biset $G$ is isomorphic to the transitive biset $H\x_{H\cap K} K$.
This result for groups inspired Park-Ragnarsson-Stancu to conjecture that $\cal F\approx\cal H\cal K$ is equivalent to a similar relation between the characteristic idempotents:
\begin{equation}\label{eqPRS}
[R,id]_S^R \circ \omega_{\cal F} \circ [T,id]_T^S = \omega_{\cal H}\circ [R\cap T,id]_T^R\circ \omega_{\cal K}
\end{equation}
Thanks to Theorem \ref{thmCharIdemMultiplicationIntro} and its corollary we can now directly calculate under which circumstances \eqref{eqPRS} holds, which results in the following theorem.
\begin{mainthm}\label{thmCompositionProducts}
Let $\cal F$ be a saturated fusion system on a $p$-group $S$, and suppose that $\cal H,\cal K$ are saturated fusion subsystems of $\cal F$ on subgroups $R,T\leq S$ respectively.

Then the characteristic idempotents satisfy
\begin{equation}\label{eqCompositionIdems}
[R,id]_S^R \circ \omega_{\cal F} \circ [T,id]_T^S = \omega_{\cal H}\circ [R\cap T,id]_T^R\circ \omega_{\cal K} \quad\text{in $A(T,R)_{(p)}$}
\end{equation}
if and only if $\cal F\approx\cal H\cal K$ and for all $Q\leq R\cap T$ we have
\begin{equation}\label{eqCompositionHoms}
\abs{\cal F(Q,S)} = \frac{\abs{\cal H(Q,R)}\cdot \abs{\cal K(Q,T)}}{\abs{\Hom_{\cal H\cap \cal K}(Q,R\cap T)}}.
\end{equation}
\end{mainthm}

\begin{proof}%[Proof of Theorem \ref{thmCompositionProducts}]
We first analyse each side of \eqref{eqCompositionIdems} separately in order to ease the proof of the theorem.

The element $l.h.s. := [R,id]_S^R \circ \omega_{\cal F} \circ [T,id]_T^S$ is the characteristic idempotent for $\cal F$ restricted to $A(T,R)_{(p)}$. For subgroups $D\leq R\x T$ we therefore have
\[\Phi_{D}([R,id]_S^R \circ \omega_{\cal F} \circ [T,id]_T^S)=0\]
unless $D$ has the form $\Delta(P,\ph)$ with $P\leq T$ and $\ph\in \cal F(P,R)$, and for such $P$ and $\ph$ we get
\[\Phi_{\Delta(P,\ph)}([R,id]_S^R \circ \omega_{\cal F} \circ [T,id]_T^S) = \frac{\abs{S}}{\abs{\cal F(P,S)}}.\]
For the right hand side $r.h.s. :=\omega_{\cal H}\circ [R\cap T,id]_T^R\circ \omega_{\cal K}$ we know from Corollary \ref{corCharDoubleStabilization} that $r.h.s.\in A(\cal K,\cal H)_{(p)}$ corresponds to the basis element $\beta_{\Delta(R\cap T,id)}$ in $A(\cal H\x \cal K)_{(p)}$. Hence we have
\[\Phi_D(\omega_{\cal H}\circ [R\cap T,id]_T^R\circ \omega_{\cal K})=0\]
unless $D$ is $(\cal H\x \cal K)$-conjugate to a subgraph of $\Delta(R\cap T,id)$, i.e., $D$ has the form $\Delta(\rho Q,\psi\rho^{-1})$ with $Q\leq R\cap T$, $\rho\in\cal K(Q,T)$, and $\psi\in\cal H(Q,R)$. If $D$ has this form, then we get
\begin{align*}
\Phi_{\Delta(\rho Q,\psi\rho^{-1})}(r.h.s.)&=\Phi_{\Delta(\rho Q,\psi\rho^{-1})}(\omega_{\cal H}\circ [R\cap T,id]_T^R\circ \omega_{\cal K})
\\ &= \Phi_{\Delta(\rho Q,\psi\rho^{-1})}(\beta_{\Delta(R\cap T,id)})
\\ &=  \Phi_{\Delta(Q,id)}(\beta_{\Delta(R\cap T,id)})\hspace{2cm} \text{by $(\cal H\x \cal K)$-stability}
\\ &= \frac{\abs{\Hom_{\cal H \x \cal K}(\Delta(Q,id),\Delta(R\cap T,id))}\cdot \abs{R\x T}}{\abs{\Delta(R\cap T,id)}\cdot \abs{\Hom_{\cal H\x \cal K}(\Delta(Q,id),R\x T)}}\quad \text{by Remark \ref{rmkBetas}}
\\ &= \frac{\abs{\Hom_{\cal H \cap \cal K}(Q,R\cap T)}\cdot \abs{R}\cdot\abs{T}}{\abs{R\cap T}\cdot \abs{\cal H(Q,R)}\cdot \abs{\cal K(Q,T)}}
\\ &= \abs{RT}\cdot \frac{\abs{\Hom_{\cal H \cap \cal K}(Q,R\cap T)}}{\abs{\cal H(Q,R)}\cdot \abs{\cal K(Q,T)}}.
\end{align*}
Suppose that \eqref{eqCompositionIdems} is true. Comparing $\Phi_{\Delta(1,id)}(l.h.s.)=\abs S$ and $\Phi_{\Delta(1,id)}(r.h.s.)=\abs{RT}$, we see that we must have $\abs{S}=\abs{RT}$, and consequently $S=RT$. Furthermore we know that if $P\leq T$ and $\ph\in \cal F(P,R)$, then $\Phi_{\Delta(P,\ph)}(l.h.s.)\neq 0$. It is therefore a requirement for \eqref{eqCompositionIdems} that $\Phi_{\Delta(P,\ph)}(r.h.s.)\neq 0$ as well whenever $P\leq T$ and $\ph\in \cal F(P,R)$. This is the case exactly when $\Delta(P,\ph)$ has the form $\Delta(\rho Q,\psi\rho^{-1})$ with $\rho\in \cal K$ and $\psi\in \cal H$, hence $\ph=\psi\rho^{-1}\in \cal H\cal K$, so we must have $\cal F\approx\cal H\cal K$. Because $S=RT$, the equality $\Phi_{\Delta(Q,id)}(l.h.s.)=\Phi_{\Delta(Q,id)}(r.h.s)$ additionally gives us \eqref{eqCompositionHoms}.

If we conversely suppose that $\cal F\approx\cal H\cal K$, then $\Phi_{\Delta(P,\ph)}(l.h.s.)$ and $\Phi_{\Delta(P,\ph)}(r.h.s.)$ are non-zero for the same indices $\Delta(P,\ph)$, and because $S=RT$, the only obstacle for equality of fixed points $\Phi_{\Delta(\rho Q,\psi\rho^{-1})}(l.h.s.)=\Phi_{\Delta(\rho Q,\psi\rho^{-1})}(r.h.s.)$ is whether it holds that
\[\frac1{\abs{\cal F(Q,S)}}=\frac{\abs{\Hom_{\cal H \cap \cal K}(Q,R\cap T)}}{\abs{\cal H(Q,R)}\cdot \abs{\cal K(Q,T)}}\]
for all $Q\leq R\cap T$, which is \eqref{eqCompositionHoms}.
\end{proof}

\begin{ex}\label{exA6product}
The following example shows that the conjecture of Park-Ragnarsson-Stancu fails in general. We consider the alternating group $A_6$, and identify one of its Sylow $2$-subgroups with the dihedral group $D_8$. The associated fusion system $\cal F:=\cal F_{D_8}(A_6)$ is the saturated fusion system on $D_8$ wherein all five subgroups of order $2$ are conjugate. Let $R,T\leq D_8$ be the two Klein four-groups inside $D_8$, and let $\cal H=\cal F_R(R\rtimes \Z/3)$, $\cal K=\cal F_T(T\rtimes \Z/3)$ be fusion subsystems of $\cal F$ on $R$ and $T$ respectively, with $\Z/3$ acting nontrivially on $R\cong T\cong \Z/2\x \Z/2$. Then $\cal H$ and $\cal K$ both contain the order 3 automorphisms of the Klein four-group, and both are saturated.

We claim that $\cal F\approx\cal H\cal K$. First of all $D_8=RT$ is clear. Next, there is no isomorphism between $R$ and $T$ in $\cal F$, so the only subgroups of $T$ that map to $R$ in $\cal F$, are the subgroups of order $2$ and the trivial group. Suppose $A\leq T$ has order $2$. Then every morphism $\ph\in \cal F(A,R)$ factors through $Z(D_8) = R\cap T$, and can therefore be factored as $\ph=\rho\psi$ with $\psi\in \cal K(A,Z(D_8))$ and $\rho\in \cal H(Z(D_8),R)$. Hence we have $\cal F\approx\cal H\cal K$.

However \eqref{eqCompositionHoms} fails for the intersection $Q:=Z(D_8)=R\cap T$ since we get
\[\abs{\cal F(Z(D_8),S)} = 5 \neq \frac{3\cdot 3}1 = \frac{\abs{\cal H(Z(D_8),R)}\cdot \abs{\cal K(Z(D_8),T)}}{\abs{(\cal H\cap \cal K)(Z(D_8),Z(D_8))}}.\]
Consequently, we have $\cal F\approx\cal H\cal K$ but not \eqref{eqCompositionIdems}, so the conjecture of Park-Ragnarsson-Stancu is false in general.
\end{ex}

The following generalization of \cite{ParkRagnarssonStancu}*{Theorem 1.3} is an example of how to apply Theorem \ref{thmCompositionProducts} to prove a special case of the conjecture. The special case proved by Park-Ragnarsson-Stancu requires that $R=S$, but the proposition below does not have this requirement.
\begin{prop}\label{propNormalCompositionProduct}
Let $\cal F$ be a saturated fusion system on a $p$-group $S$, and let $\cal H,\cal K$ be saturated fusion subsystems of $\cal F$ on subgroups $R,T\leq S$ respectively. Suppose that $\cal K$ is weakly normal in $\cal F$, i.e., $\cal K$ is saturated and $\cal F$-invariant in the sense of \cite{Aschbacher}.

Then $\cal F\approx\cal H\cal K$ if and only if the characteristic idempotents satisfy
\[[R,id]_S^R \circ \omega_{\cal F} \circ [T,id]_T^S = \omega_{\cal H}\circ [R\cap T,id]_T^R\circ \omega_{\cal K}.\]
\end{prop}

\begin{proof}
By Theorem \ref{thmCompositionProducts} it is sufficient to prove that $\cal F\approx\cal H\cal K$ implies \eqref{eqCompositionHoms}, so suppose $\cal F\approx\cal H\cal K$. The subsystem $\cal K$ being $\cal F$-invariant means that $T$ is strongly closed in $\cal F$, and whenever we have $A,B\leq P\leq T$ and $\ph\in\cal F(P,T)$, conjugation by $\ph$ induces a bijection $\cal K(A,B) \xto{\ph(-)\ph^{-1}} \cal K(\ph A,\ph B)$.

According to \cite{Aschbacher}*{Lemma 3.6}, the intersection $\cal H\cap \cal K$ is an $\cal H$-invariant fusion system on $R\cap T$. Suppose we have subgroups $Q\leq R\cap T$ and $Q'\sim_{\cal H} Q$, and choose an isomorphism $\ph\in \cal H(Q,Q')$. Because $T$ is strongly closed in $\cal F$, $R\cap T$ is strongly closed in $\cal H$, hence $Q'\leq R\cap T$.  By the Frattini property of $\cal H$-invariant subsystems, \cite{Aschbacher}*{Section 3}, $\ph$ can be factored as $\ph=\eta\kappa$ with $\kappa\in (\cal H\cap \cal K)(Q,R\cap T)$ and $\eta\in \Aut_{\cal H}(R\cap T)$. If we let $Q'':=\kappa(Q)$, we then have $\abs{\cal K(Q,T)} = \abs{\cal K(Q'',T)}$ and $\abs{(\cal H\cap \cal K)(Q,T)} = \abs{(\cal H\cap \cal K)(Q'',T)}$. Furthermore, the $\cal H$-isomorphism $\eta\colon Q''\to Q'$ is defined on all of $R\cap T$, so the $\cal F$-stability of $\cal K$ and $\cal H$-stability of $\cal H\cap \cal K$ implies that $\eta$ induces bijections $\cal K(Q'',T) \cong \cal K(Q',T)$ and $(\cal H\cap \cal K)(Q'',T) \cong (\cal H\cap \cal K)(Q',T)$.

We will now prove \eqref{eqCompositionHoms}, and because $T$ is strongly closed in $\cal F$, we have $\abs{\cal H(Q,R)} = \abs{\cal H(Q,R\cap T)}$, so we must show
\[\abs{\cal F(Q,T)} = \frac{\abs{\cal H(Q,R\cap T)}\cdot \abs{\cal K(Q,T)}}{\abs{(\cal H\cap \cal K)(Q,R\cap T)}}\]
for all $Q\leq R\cap T$. Let therefore $Q\leq R\cap T$ be given. For every homomorphism $\ph\in \cal F(Q,T)$, we can factor $\ph^{-1}\colon \ph Q \to Q$ as $\ph^{-1}=\eta^{-1}\kappa^{-1}$ with $\eta^{-1}\in \cal H$ and $\kappa^{-1}\in \cal K$, or equivalently $\ph=\kappa\eta$. We will enumerate $\cal F(Q,T)$ by counting the number of pairs of isomorphisms $(\kappa,\eta)$ with $\eta\colon Q\to Q'$ in $\cal H$ and $\kappa\colon Q'\to Q''$ in $\cal K$. The number of choices for $\eta$ is $\abs{\cal H(Q,R\cap T)}$, and for each $\eta\colon Q\to Q'$ the number of choices for $\kappa$ is $\abs{\cal K(Q',T)}$. Because $Q'$ is isomorphic to $Q$ in $\cal H$, the arguments above imply that $\abs{\cal K(Q',T)}=\abs{\cal K(Q,T)}$, which is independent of the chosen $\eta\in \cal H(Q,R\cap T)$. The total number of composable pairs $(\kappa,\eta)$ is therefore \[\abs{\cal H(Q,T)}\cdot\abs{\cal K(Q,T)}.\]
Given a pair $(\kappa,\eta)$ of composable isomorphisms $Q\xto{\eta} Q' \xto{\kappa} Q''$, we then count the number of other pairs $Q \xto{\eta'} Q''' \xto{\kappa'} Q''$ that represent the same isomorphism in $\cal F$. If $(\kappa,\eta)$ and $(\kappa',\eta')$ give the same isomorphism $Q\to Q''$ in $\cal F$, then we have $\kappa\eta=\kappa'\eta'$ or equivalently $(\kappa')^{-1}\kappa = \eta'\eta^{-1}\in (\cal H\cap \cal K)(Q',R\cap T)$. Conversely, given any $\rho\in (\cal H\cap \cal K)(Q',R\cap T)$, the pair $(\kappa\rho^{-1},\rho\eta)$ defines the same $\cal F$-homomorphism as $(\kappa,\eta)$. The number of pairs representing the same map as $(\kappa,\eta)$ is therefore $\abs{(\cal H\cap \cal K)(Q',R\cap T)} = \abs{(\cal H\cap \cal K)(Q,R\cap T)}$, which is independent of the chosen pair $(\kappa,\eta)$. Hence there are $\abs{(\cal H\cap \cal K)(Q,R\cap T)}$ pairs representing each homomorphism $\ph\in \cal F(Q,T)$, so we get
\[\abs{\cal F(Q,T)} = \frac{\abs{\cal H(Q,R\cap T)}\cdot \abs{\cal K(Q,T)}}{\abs{(\cal H\cap \cal K)(Q,R\cap T)}}\]
as we wanted.
\end{proof}

Theorem \ref{thmCompositionProducts} together with Example \ref{exA6product} seem to say that the definition of weak composition product $\cal F\approx\cal H\cal K$ by Park-Ragnarsson-Stancu is too lenient. Theorem \ref{thmCompositionProducts} furthermore suggests that \eqref{eqCompositionHoms} should be an additional requirement in the definition:

\begin{dfn}\label{dfnRevisedCompositionProd}[Revised definition of composition products]
Let $\cal F$ be a fusion system on a $p$-group $S$, and let $\cal H,\cal K$ be fusion subsystems on subgroups $R,T\leq S$ respectively. We say that \emph{$\cal F$ is the composition product of $\cal H$ and $\cal K$}, written $\cal F=\cal H\cal K$, if
\begin{enumerate}
\item $S=RT$,
\item\label{dfnCompositionFactorize} for all subgroups $P\leq T$, every morphism $\ph\in \cal F(P,R)$ can be written as $\ph=\psi\rho$ with $\psi\in\cal H$ and $\rho\in\cal K$,
\item\label{dfnCompositionHoms} for all subgroups $Q\leq R\cap T$ we have
    \[\abs{\cal F(Q,S)} = \frac{\abs{\cal H(Q,R)}\cdot \abs{\cal K(Q,T)}}{\abs{\Hom_{\cal H\cap \cal K}(Q,R\cap T)}}.\]
\end{enumerate}
\end{dfn}

\begin{rmk}\label{rmkRevisedCompositionProd}
We can interpret the additional property \ref{dfnCompositionHoms} in the following way: Given $Q\leq R\cap T$, we know from \ref{dfnCompositionFactorize} that every map $\ph\in \cal F(Q,R)$ can be factored as
\[\ph=\psi\rho\colon Q \xto[\quad\cong\quad]{\rho\,\in\,\cal K} P \xto[\quad\cong\quad]{\psi\,\in\,\cal H} N \leq R.\]
There are several such factorizations of $\ph$, and the ambiguity is precisely given by the set $\Hom_{\cal H\cap \cal K}(P,R\cap T)$, where each $\chi\in \Hom_{\cal H\cap \cal K}(P,R\cap T)$ gives the alternative factorization $\ph=(\psi\chi^{-1})(\chi\rho)$. Ideally the set $\Hom_{\cal H\cap \cal K}(P,R\cap T)$ would have the same size as $\Hom_{\cal H\cap \cal K}(Q,R\cap T)$ though $Q$ and $P$ are only conjugate in $\cal K$ and not in $\cal H\cap \cal K$.

What property \ref{dfnCompositionHoms} seems to say is that similar factorizations exist for all $\ph\in \cal F(Q,S)$ and not just those that land in $R$, i.e. that we can factorize $\ph$ as a map $Q\xto{\rho} P$ in $\cal K$ followed by a map $P\xto{\tilde \psi} N$ that is in some way ``parallel'' to a map $\psi\in \cal H$. In diagram form this would look like
\[
\begin{tikzpicture}[xscale=3,yscale=1.5]
\node (Q) at (0,0) {$Q$};
\node (P) at (1,1) {$P$};
\node (N) at (2,0) {$N$,};
\node (M) at (1,-1) {$M$};

\draw[->,arrow,auto]
    (Q) edge node{$\rho\in\cal K$} (P)
        edge node[swap]{$\psi\in \cal H$} (M)
        edge node{$\ph$} (N)
    (P) edge node{$\tilde\psi$} (N)
;
\draw[dashed] (M) -- (N);
\end{tikzpicture}
\]
and the ambiguity of these decompositions would lie in $\Hom_{\cal H\cap \cal K}(Q,R\cap T)$.
Ideally we would have a pairing
\[\cal H(Q,R) \x \cal K(Q,T) \to \cal F(Q,S)\]
that is surjective and where each fiber is in bijection with $\Hom_{\cal H\cap \cal K}(Q,R\cap T)$. The proof of Proposition \ref{propNormalCompositionProduct} runs somewhat along these lines, but it is currently unknown to the author where such a pairing is in any way possible in general.
\end{rmk}

With the revised Definition \ref{rmkRevisedCompositionProd} of composition product, Theorem \ref{thmCompositionProducts} becomes
\theoremstyle{plain}
\newtheorem{mainthmZ}{Theorem}
\renewcommand{\themainthmZ}{\Alph{mainthmZ}'}
\setcounter{mainthmZ}{4}
\begin{mainthmZ}\label{thmRevisedCompositionProducts}
Let $\cal F$ be a saturated fusion system on a $p$-group $S$, and suppose that $\cal H,\cal K$ are saturated fusion subsystems of $\cal F$ on subgroups $R,T\leq S$ respectively.

Then $\cal F=\cal H\cal K$ (with respect to the revised definition) if and only if the characteristic idempotents satisfy
\begin{equation*}
[R,id]_S^R \circ \omega_{\cal F} \circ [T,id]_T^S = \omega_{\cal H}\circ [R\cap T,id]_T^R\circ \omega_{\cal K}.
\end{equation*}
\end{mainthmZ}

\makeatletter
\def\eprint#1{\@eprint#1 }
\def\@eprint #1:#2 {%
    \ifthenelse{\equal{#1}{arXiv}}%
        {\href{http://front.math.ucdavis.edu/#2}{arXiv:#2}}%
        {\href{#1:#2}{#1:#2}}%
}
\makeatother

%\clearpage

\begin{comment}
\makeatletter
\providecommand\@dotsep{5}
\def\listtodoname{List of Todos}
\def\listoftodos{\@starttoc{tdo}\listtodoname}
\makeatother
\listoftodos
\end{comment}

\end{document}